\documentclass[12pt]{amsart}

\usepackage{graphicx,color}
\usepackage{setspace}
\usepackage{amsmath, amssymb,amsthm, amsfonts}

\usepackage{accents}

\newtheorem{thm}{Theorem}[section]
\newtheorem{lma}{Lemma}[section]

\theoremstyle{definition}
\newtheorem{definition}{Definition}[section]

\theoremstyle{remark}
\newtheorem{remark}{Remark}[section]

\numberwithin{equation}{section}

\newcommand{\tr}{\mbox{tr}}

\renewcommand{\div}{\mbox{div}}

\newcommand{\Ric}{\mbox{Ric}}
\newcommand{\R}{\mathbb R}

\newcommand{\be}{\begin{equation}}
\newcommand{\ee}{\end{equation}}
\newcommand{\bee}{\begin{equation*}}
\newcommand{\eee}{\end{equation*}}

\def\p{\partial}

\def\la{\langle}
\def\ra{\rangle}
\def\lf{\left}
\def\ri{\right}
\def\Pi{\displaystyle{\mathbb{II}}}

\def\S{\Sigma}

\def\H{\mathbb{H}}

\def\vh{\vspace{.2cm}}

\def\m{\mathfrak{m}}

\def\l{\lambda}
\def\e{\epsilon}

\def\wt{\widetilde }

\def\ve{\varepsilon}

\def\k{\kappa}
\def\ko{\kappa^{(0)}}

\def\ringA{\accentset{\circ}{A}}

\begin{document}
\title[]
{Quasi-local mass integrals and the \\  total mass}

\author{Pengzi Miao$^1$}
\address[Pengzi Miao]{Department of Mathematics, University of Miami, Coral Gables, FL 33146, USA.}
\email{pengzim@math.miami.edu}
\thanks{$^1$Research partially supported by Simons Foundation Collaboration Grant for Mathematicians \#281105.}

\author{Luen-Fai Tam$^2$}
\address[Luen-Fai Tam]{The Institute of Mathematical Sciences and Department of
 Mathematics, The Chinese University of Hong Kong, Shatin, Hong Kong, China.}
 \email{lftam@math.cuhk.edu.hk}
\thanks{$^2$Research partially supported by Hong Kong RGC General Research Fund \#CUHK 14305114}
\author{Naqing Xie$^3$}
\address[Naqing Xie]{School of Mathematical Sciences, Fudan
University, Shanghai 200433, China.}
\email{nqxie@fudan.edu.cn}
\thanks{$^3$Research partially supported by the National Science Foundation of China \#11421061}
\renewcommand{\subjclassname}{
  \textup{2010} Mathematics Subject Classification}


\begin{abstract}
On asymptotically flat and asymptotically hyperbolic manifolds,
   by evaluating the total mass via the Ricci tensor,
we show that the   limits  of certain Brown-York type
and Hawking type
quasi-local mass integrals
equal  the total mass of the manifold   in all dimensions.
\end{abstract}


\maketitle

\markboth{Pengzi Miao, Luen-Fai Tam and Naqing Xie}
{Quasi-local mass integrals and the total mass}

\section{introduction}
In this work, we
discuss the relation between certain quasi-local mass integrals and
the total mass of asymptotically flat and asymptotically hyperbolic manifolds.

First we  introduce some notations  and recall some definitions.
Let $(M^n,g)$ be a Riemannian manifold of dimension $ n \ge 3$.
For $ \l = 0$  or $ -1$,  let
\be\label{e-Gg-intro}
G^g_\lambda=\Ric(g)-\frac 12 \lf[ \mathcal{S}_g -\lambda(n-1)(n-2) \ri] g ,
\ee
where     $\Ric(g)$ and $\mathcal{S}_g$ are the Ricci tensor and the scalar curvature of $g$,  respectively.
 
Throughout the paper, $b_n$ and $c_n$ will  always denote the constants:
\be\label{e-bc}
\left\{
  \begin{array}{ll}
    b_n= &  \frac1{ 2 (n-1)  \omega_{n-1}}, \\
     c_n= &   \frac{1}{(n-1)(n-2)\omega_{n-1}},
  \end{array}
\right.
\ee
where $\omega_{n-1}$ is the volume  of the unit sphere $ \mathbb{S}^{n-1} $  in $\R^n$. 

\begin{definition}\label{def-AF}
$(M^n,g)$ is an {\it asymptotically flat} (AF) manifold  if, outside a compact set,
$M^n$ is diffeomorphic to  $ \R^n  \setminus \{ | x | \le r_0 \}$ for some $r_0>0$
and, with respect to the coordinates $ x = (x^1, \ldots, x^n )$ on $ \R^n$, the metric components $\{ g_{ij} \} $  satisfy
\be \label{eq-AF}
g_{ij} - \delta_{ij} = O ( | x |^{- \tau} ) ,  \p g_{ij} = O ( |x|^{-1-\tau}), \  \p^2 g_{ij} = O (|x|^{-2 - \tau} )
\ee
for some $\tau>\frac {n-2}{2}$. Here $ \p $ denotes the partial differentiation on $ \R^n$.
Moreover, one assumes the scalar curvature $\mathcal{S}_g$  is  in $L^1(\R^n)$.
\end{definition}

On an asymptotically flat $(M^n, g)$, the total mass  (or  the ADM mass) $\m$ (\cite{ADM61})  is defined as:

\be\label{e-ADMmass}
\m=b_n\lim_{r\to\infty}\int_{S_r}\lf(g_{ij,i}-g_{ii,j}\ri)\nu^j_ed\sigma_e ,
\ee
where $S_r=\{|x|=r\}$,  $\nu_e$ is the unit outward normal  and $d\sigma_e$ is the area element on $S_r$  both with respect to the Euclidean metric $g_e$.

The definitions  of an asymptotically hyperbolic manifold and its mass  are more 
elaborate (cf. \cite{ChruscielHerzlich, Wang, Zhang}). 
Here we use the most general form given in \cite{ChruscielHerzlich}. 
 Let $\H^n$  denote   the standard hyperbolic space.
The  metric $g_0$ on $\H^n$ can be written as 
$$
g_0=d\rho^2+ (\sinh \rho)^2  h_0,
$$
where  $\rho$ is  the $g_0$-distance to a fixed point $o$ and
$h_0$ is  the standard metric on $\mathbb{S}^{n-1}$.
Let $\ve_0=\p_\rho$, $\ve_\alpha =\frac1{\sinh\rho}f_\alpha $, $\alpha =1,\dots,n-1$, where $\{f_\alpha \}_{1\le \alpha \le n-1}$ is a local orthonormal
frame on $\mathbb{S}^{n-1}$,
then $\{ \ve_i \}_{0 \le i \le n-1 }$ form a local  orthonormal  frame on $\H^n$.
By abuse of notation, we  let $\S_\rho$ denote the geodesic sphere   in $\H^n$  of radius $\rho$ centered at $o$.

\begin{definition}\label{d-AH}  $(M^n,g)$ is called  {\it asymptotically hyperbolic} (AH)
if, outside a compact set, $M$ is diffeomorphic to the exterior of  some geodesic sphere  $\S_{\rho_0}$  in $\H^n$ such that
the metric components  $ g_{ij} =  g(\ve_i,\ve_j)$,  $0 \le i, j \le n-1$,  satisfy
\be
|g_{ij}-\delta_{ij}|=O(e^{-\tau \rho}),\ |\ve_k(g_{ij})|=O(e^{-\tau \rho}), \ |\ve_k(\ve_l(g_{ij}))|=(e^{-\tau \rho})
\ee
for some $\tau>\frac n2$. Moreover, one assumes $\mathcal{S}_g+n(n-1)$ is in
$L^1(\H^n, e^\rho dv_0)$,  where  $dv_0$ is the volume element of $g_0$ and $\mathcal{S}_g$ is the scalar curvature of $g$.
\end{definition}

On an asymptotically hyperbolic $(M^n, g)$,
by the results in \cite{ChruscielHerzlich},
  its mass $\mathbf{M}(g)$ is  a linear function  on the kernel of $(D\mathcal{S})_{g_0}^*$,
 where  $ (D\mathcal{S})_{g_0}^* $ is  the formal adjoint of the linearization of the scalar curvature  at $g_0$.
   Let $ \theta=(\theta^1,\dots,\theta^n)\in \mathbb{S}^{n-1}\subset \R^n$,
    the functions
   \be \label{e-def-Vi}
   V^{(0)}=\cosh \rho , \ \ V^{(j )}=\theta^j  \sinh \rho , \ 1 \le j \le n
   \ee
   then form a basis of   the kernel of  $(D\mathcal{S})_{g_0}^*$.
$ \mathbf{M} (g)$ is  defined  by 
  \be
  \begin{split}
 \mathbf{M} (V^{(i)} )  
= & \  b_n  \lim_{\rho \rightarrow \infty} \int_{\Sigma_\rho} \lf[ V^{(i)} ( \div_0 h  - d \tr_0 h ) \ri. \\
& \ \hspace{2cm}  \lf. - h (\nabla_0 V^{(i)}, \cdot) 
+ ( \tr_0 h ) d V^{(i)} \ri] ( \tilde \nu)  d \sigma_0 , 
  \end{split}
  \ee
  where $h= g - g_0$, $\tilde \nu$ is the $g_0$-unit outward normal to $\Sigma_\rho$, $d \sigma_0$ 
  is the volume element on $\Sigma_\rho$ of the metric induced from $g_0$, 
  and $\div_0$, $ \tr_0$, $\nabla_0$ denote the divergence, trace, the covariant derivative
  with respect to $g_0$, respectively. 
  We refer readers to  \cite{ChruscielHerzlich} (also see \cite{Michel}) 
  for a well explained motivation behind this definition. 
   
For our purpose in this work, we make use of
 the following formulae,
available in the 
literature,
which compute   $ \m $  and
$ \mathbf{M}(g)(V^{(i)}) $  in terms of the tensor $G_\l^g$ defined in \eqref{e-Gg-intro}.

  \begin{thm}\label{t-mass}
  \begin{enumerate}

                             \item [(i)] Suppose $(M^n,g)$ is an AF manifold, then
                             $$
                             \m= - c_n\lim_{r\to\infty}\int_{S_r}G^g_0 (X,\nu_g)d\sigma_g.
                             $$
                            Here $\nu_g$ is unit outward normal to $S_r$ in $(M^n, g)$,
                             $d\sigma_g$ is the volume element of $S_r$ with
                             respect to the metric induced by $g$,     $X=\sum_{i=1}^nx^i\frac{\p}{\p x^i}$  is the  conformal  radial Killing vector field on
                             $\R^n$.

 \item [(ii)]  Suppose $(M^n,g)$ is an AH  manifold.
 Then,  for each $ 0 \le i\le n$,
$$
\mathbf{M}(g) (V^{ (i)})= - c_n\lim_{ \rho  \to\infty}\int_{\S_\rho}G^g_{-1} (X^{(i)},\nu_g)d\sigma_g .
$$
Here $\{ X^{(i)} \}$ are the conformal Killing vector fields  on $\H^n$,   given by
$$ X^{(0)}=x^k\frac{\p}{\p x^k}, \ \   X^{(j)}=\frac{\p}{\p x^j}, \ 1\le j \le n  $$
in the ball model $ (B^n, g_0)$ of $\H^n$ where
${B}^n=\{  x \in \R^n|\ |x|<1\}$ and
$  (g_0)_{ij}=\frac{4}{(1-|x|^2)^2} \delta_{ij} $,
$ \{ V^{(i)} \}$ are functions on $\H^n$ defined by \eqref{e-def-Vi}, which on $(B^n, g_0)$  take the form of
$$
 V^{(0)} = \frac{1+r^2}{1-r^2} , \ \ V^{(i)} = \frac{2 x^i}{1-r^2}  , \ 1 \le i \le n
$$
 with $ r = | x | $, $\S_\rho $ is  the  geodesic sphere  of radius $\rho$  centered at $x=0$  in $({B}^n, g_0)$,
$\nu_g$ is  the outward unit normal to  $\S_\rho $ with respect to  $g$,  and $d\sigma_g$ is the volume element on $\S_\rho $
of the metric induced by $g$.

 \end{enumerate}

  \end{thm}

 We emphasize that Theorem \ref{t-mass} (i) is
widely known in the relativistic community (cf. \cite{Ashtekar-Hansen, Chrusciel})
 and a  proof of it can be found
in \cite{Huang-ICCM, MiaoTam2015, Herzlich, WangWu2015}.
Theorem \ref{t-mass} (ii) was recently proved by Herzlich \cite{Herzlich}.

In this paper, we give applications of  the formulae in Theorem \ref{t-mass}. 
Our main results  are the following:

\begin{thm}  \label{t-intro-1}
Let $(M^n,g)$ be  an asymptotically flat manifold.
Let $ \{ \Sigma_\rho \}$ be a family of nearly round hypersurfaces in $(M^n, g)$ (see Definition \ref{def-nearly-round}).
Then
\be \label{eq-thm-S-H-intro}
 \lim_{\rho \to \infty}  \frac{c_n}{2}  \lf(\frac{|\Sigma_\rho| }{\omega_{n-1}}\ri)^{\frac1{n-1}}
 \int_{\S_\rho}   \lf(    S^\rho - \frac{n-2}{n-1} H^2   \ri) d\sigma_\rho  =  \m  .
 \ee
If in addition  $ \Sigma_\rho $ can be isometrically embedded in $ \R^n$
when $\rho$ is   sufficiently large, which is automatically satisfied if  $n=3$,
then
\be \label{eq-thm-BY-intro}
  \lim_{\rho \to \infty}  2 b_n \int_{\S_\rho}  (H_0 - H )d\sigma_\rho   =  \m
\ee
and
\be \label{eq-thm-sigma-2-intro}
  \lim_{\rho \to \infty}  c_n    \lf(\frac{|\Sigma_\rho| }{\omega_{n-1}}\ri)^{\frac1{n-1}}
   \int_{\S_\rho} \sum_{\alpha < \beta} \lf( \ko_\alpha \ko_\beta  - \kappa_\alpha \kappa_\beta \ri) d \sigma_\rho
=  \m   .
\ee
Here  $ | \Sigma_\rho |$, $ S^\rho$, $ d \sigma_\rho$
 are  the volume, the scalar curvature, the volume element   of $\Sigma_\rho$, respectively;
 $H$, $ \{ \kappa_\alpha\}$  are the mean curvature, the principal curvatures  of $\Sigma_\rho$ in $ (M^n ,g )$,  respectively;  $H_0$,  $ \{ \ko_\alpha\}$  are the mean curvature,  the principal curvatures
  of the isometric embedding of $\Sigma_\rho$ in $ \R^n$, respectively;
 and $\m $ is the total mass of $(M^n, g)$.
\end{thm}

\begin{thm}\label{t-intro-2}
Let  $(M^n,g)$ be an asymptotically hyperbolic manifold with suitable decay
rate $\tau$ (specified in  Theorem \ref{t-AH-1}).
Let $ \{ \Sigma_\rho \}_{\rho \ge \rho_0}  $ be  the hypersurfaces in $M^n$ which are  the
  geodesic spheres  of radius $\rho$  in  $\H^n$ as specified  in Definition \ref{d-AH}.
   Let $\{ V^{(i)} \}_{0\le i\le n-1} $ be the functions given in \eqref{e-def-Vi}.
  Then the following are true:
\begin{enumerate}
  \item [(I)]
   \begin{itemize}
   \item[(a)]   \bee
\begin{split}
& \ \mathbf{M}(g) (V^{ (0)}) \\
=  & \   \lim_{ \rho \rightarrow \infty}   \frac{c_n}{2}
 \lf( \frac{ | \Sigma_\rho |  }{ \omega_{n-1} } \ri)^\frac{1}{n-1}
  \int_{\Sigma_\rho}  \lf[    {S}^\rho  -  \frac{n-2}{n-1} H^2  + (n-1)(n-2)\ri]  d\sigma_\rho     .
\end{split}
\eee
    \item[(b)]
    For $1\le i\le n$,
    \bee
\begin{split}
& \ \mathbf{M}(g) (V^{ (i)}) \\
=  & \   \lim_{ \rho \rightarrow \infty}  \frac{c_n}{2}    \lf( \frac{ | \Sigma_\rho |  }{ \omega_{n-1} } \ri)^\frac{1}{n-1}
    \int_{\Sigma_\rho} \frac{x^i}{|x|} \lf[  {S}^\rho  - \frac{n-2}{n-1} H^2 + (n-1)(n-2) \ri]  d\sigma_\rho     ,
\end{split}
    \eee
where $\{ x^i \} $ are the coordinate functions in the ball model  $(B^n, g_0)$ of $\H^n$.
  \end{itemize}

  \vh

  \item [(II)]
Suppose    $\Sigma_\rho $ can be isometrically embedded in $\H^n$ for sufficiently  large $\rho$,
 which is automatically satisfied   when $n=3$.
 \begin{enumerate}
\item[(a)]
For each $0 \le i \le n$,
$$ \mathbf{M}(g) ( V^{(i)} ) =  2 b_n \lim_{\rho \rightarrow \infty} \int_{\Sigma_\rho} (H_0 - H) V^{(i)} \, d \sigma_\rho .  $$

\item[(b)] Consider the hyperboloid model of $ \H^n$, i.e.
$$
\{(x^0,x^1,\dots,x^n)\in \R^{n,1}|\ (x^0)^2-\sum_{i=1}^n(x^i)^2=1, x^0>0\} ,
$$
where $ \R^{n,1}$ is the $(n+1)$-dimensional Minkowski space,
 there exist isometric embeddings  of $ \{  \Sigma_\rho  \} $ in $\H^n$ such  that
$$
\mathbf{M}(g)(V)=
2 b_n \lim_{\rho \to \infty}\int_{\Sigma_\rho }(H_0-H)\mathbf{x} \ d\sigma_\rho ,
$$
where
$ \mathbf{M}(g) (V)=\lf(\mathbf{M}(g) (V^{(0)}), \dots,\mathbf{M}(g) (V^{(n)})\ri)$
and $\mathbf{x} $ is   the position vector  of  the embedding  of $\S_\rho$  in $ \H^{n} \subset \R^{n,1}$.
\end{enumerate}
\end{enumerate}
Here  $ | \Sigma_\rho |$, $ S^\rho$, $ d \sigma_\rho$
 are  the volume, the scalar curvature, the volume element   of $\Sigma_\rho$, respectively;
 $H$, $ \{ \kappa_\alpha\}$  are the mean curvature, the principal curvatures  of $\Sigma_\rho$ in $ (M^n ,g )$,  respectively;  $H_0$,  $ \{ \ko_\alpha\}$  are the mean curvature,  the principal curvatures
  of the isometric embedding of $\Sigma_\rho$ in $ \H^n$, respectively;
and $ \mathbf{M}(g) (\cdot)$ is the mass function of $(M^n, g)$.

\end{thm}

We  give some remarks concerning Theorems \ref{t-intro-1}  and \ref{t-intro-2}.

\begin{remark}
When $n=3$,  \eqref{eq-thm-S-H-intro} and  \eqref{eq-thm-BY-intro}
in  Theorem \ref{t-intro-1} become
\be \label{eq-3d-H-limit}
 \m
=   {\m}_{_H} (\Sigma_\rho) +  o (1)
\ee
and
\be \label{eq-3d-BY-limit}
 \m
=   {\m}_{_{BY}} (\Sigma_\rho) +  o (1)
\ee
as $ \rho \rightarrow \infty$, respectively,
where
$$ {\m}_{_H} (\Sigma_\rho) =  \sqrt{  \frac{ | \Sigma_\rho| }{16 \pi} }   \lf[  1 - \frac{1}{16 \pi} \int_{\Sigma_\rho}  H^2 d\sigma_\rho \ri]  $$
 is the Hawking quasi-local mass  \cite{Hawking-mass} of $ \Sigma_\rho $ in $(M^3, g)$,
and
$$
\m_{_{BY} }  (\Sigma_\rho )  = \frac{1}{8\pi} \int_{\S_\rho }(H_0-H) d\sigma_\rho
$$
is the Brown-York quasi-local mass (\cite{BY1, BY2}) of $\Sigma_\rho $ in $(M^3, g)$.
In this case,  \eqref{eq-3d-BY-limit} was first   proved for coordinate spheres $\{ S_r \}$ in \cite{FanShiTam}.
Later in  \cite{ShiWangWu09},    \eqref{eq-3d-H-limit} and \eqref{eq-3d-BY-limit} were  proved  for
nearly round 2-surfaces.
(Examples of surfaces which are not nearly round,  but along which $\m_{_{BY}} (\cdot) $  converges to $\m$,  were given  in \cite{FanKwong}.)
Both proofs  in  \cite{FanShiTam, ShiWangWu09}
are   via  delicate pointwise estimates  of $H$,
together with an  application of  the Minkowski integral formula   (cf. \cite{KLBG})
to handle the integral of $H_0$.
For  $ n \ge 3$, our proof of  \eqref{eq-thm-S-H-intro} and  \eqref{eq-thm-BY-intro}
in  Theorem \ref{t-intro-1} is to use Theorem \ref{t-mass}(i) with $\{ S_r \}$
replaced by $ \{ \Sigma_\rho \}$.
 Indeed, it was proved  in  \cite[Theorem 2.1]{MiaoTam2015}
 that  Theorem \ref{t-mass}(i) is still valid if  $ \{ S_r \}$ is replaced by
 a  family of Lipschitz     hypersurfaces  $\{ \Sigma_l\}$,
which are boundaries of  domains in $\R^n$
so that $r_l:=\inf_{x\in \Sigma_l}|x|$ tends to infinity with the volume  $ | \Sigma_l | $ satisfying
 $|\Sigma_l|\le Cr_l^{n-1}$ for some   constant $C$ independent of $l$.
In particular,   the mass formula in Theorem \ref{t-mass}  (i) is   applicable   to
nearly round hypersurfaces defined in Definition \ref{def-nearly-round}.
\end{remark}

\begin{remark}
When $ n \ge 4$, though  \eqref{eq-thm-BY-intro} and \eqref{eq-thm-sigma-2-intro} in Theorem \ref{t-intro-1}
 are  proved under the assumption that
$ \Sigma_\rho $ can be isometrically embedded in  $ \R^n$ for large $\rho$,
it is directly applicable in certain  special  cases. For instance, in \cite{ShiTam02}
an asymptotically flat metric
$ g = u^2 d r^2 + g_r $
with zero scalar curvature
was constructed on the exterior  $E$ of a convex hypersurface $ \Sigma_0$ in $ \R^n$,
where $ r $ is the Euclidean distance to $ \Sigma_0$ and $ g_r $ represents the induced metric on
$\Sigma_r$ that is the level set of $r$. In this case, one easily checks that $ \{ \Sigma_r \}$
is a family of nearly round hypersurfaces in $(E, g)$,   which are automatically isometrically embedded in $ \R^n$.
 Hence,   \eqref{eq-thm-BY-intro} in Theorem \ref{t-intro-1}  implies
Theorem 2.1(c) in  \cite{ShiTam02}.
\end{remark}

\begin{remark}
Since
$  H_0 =  \sum_{\alpha}  \ko_\alpha ,  \  H = \sum_\alpha  \kappa_\alpha ,$
and
$$
\sum_{\alpha < \beta} \ \ko_\alpha \ko_\beta  , \
\sum_{\alpha < \beta}  \kappa_\alpha \kappa_\beta
$$
are often known as the seconder order mean curvature of $ \Sigma_\rho$ in $\R^n$, $(M^n, g)$ respectively,
\eqref{eq-thm-sigma-2-intro} in Theorem \ref{t-intro-1}
may be viewed as an   analogue of
\eqref{eq-thm-BY-intro} in terms of the seconder order mean curvature.
\end{remark}

\begin{remark}
When $n=3$,
part (I) (a) of Theorem \ref{t-intro-2}
becomes
\bee
\begin{split}
 \mathbf{M}(g) (V^{ (0)})
=  \tilde{\m}_{_H} (\Sigma_\rho) +  o (1) , \  \mathrm{as} \  \rho \rightarrow \infty,
\end{split}
\eee
where
$$ \tilde{\m}_{_H} (\Sigma_\rho) =  \sqrt{  \frac{ | \Sigma_\rho| }{16 \pi} }   \lf[  1 - \frac{1}{16 \pi} \int_{\Sigma_\rho}  H^2 d\sigma_\rho + \frac{| \Sigma_\rho |}{4 \pi} \ri]  $$
 is known as a hyperbolic analogue of the Hawking mass (cf. \cite{Wang, Neves}).
 However,   part (I)  of Theorem \ref{t-intro-2}  shows that  the limit of $  \tilde{\m}_{_H} (\Sigma_\rho) $ is  only part of the mass function $\mathbf{M}(g) (V) $. Therefore, one may expect a complete definition  of
 the hyperbolic Hawking mass of $\Sigma_\rho$  to involve  all  integrals under the limit sign in part (I) (a) and (b)
 of Theorem \ref{t-intro-2}.
\end{remark}

\begin{remark}
When $n=3$,
the integral
$$
\frac{1}{8 \pi} \int_{\S_\rho}(H_0-H)\mathbf{x} \ d\sigma_\rho
$$
in  part (II) (b) of Theorem \ref{t-intro-2}
is analogous to the quasi-local mass integral for closed surfaces in $3$-manifolds with $S_g \ge -6 $
considered  in \cite{WangYau2007,ShiTam2007}.
In the context of conformally compact asymptotically hyperbolic $3$-manifolds,
part (II) (b) of Theorem \ref{t-intro-2}  was proved   in \cite{KwongTam}. For any $n \ge 3$,
we  prove   Theorem \ref{t-intro-2} by applying   Theorem \ref{t-mass}(ii).
\end{remark}

The paper is organized as follows. In section 2, we prove Theorem \ref{t-intro-1}. In section 3, we prove Theorem \ref{t-intro-2}.

\section{Limits of quasi-local mass integrals in AF manifolds}

In an asymptotically flat $(M^n, g)$, given a closed hypersurface $ \Sigma$ that  is homologous to a large coordinate sphere at infinity,
we  let $ H$ and $ A$  be the mean curvature and the second fundamental form of $ \Sigma$ with respect to the infinity pointing unit normal $ \nu$ respectively,
and let  $ \ringA$ be the traceless second fundamental form of $ \Sigma$, i.e. $ \ringA  = A - \frac{H}{n-1} \sigma$, where $ \sigma$ is the induced metric on $ \Sigma$.
We also let $ d \sigma$, $ \nabla $,  $ | \Sigma | $ and $ \text{Diam}(\Sigma)$
 denote the volume form on $ \Sigma $,  the covariant differentiation on  $ \Sigma$,
 the volume and  the intrinsic diameter of $ \Sigma$, respectively.

The concept of a $1$-parameter family of {\em nearly round surfaces} at infinity of an asymptotically flat $3$-manifold was used in
\cite{ShiWangWu09}. Below we give its analogue in general dimensions.

\begin{definition} \label{def-nearly-round}
Let $ \{ \Sigma_\rho \}$ be a  family of closed, connected, embedded hypersurfaces,  homeomorphic to the $(n-1)$-dimensional sphere,
  in $(M^n, g)$ with $\rho \in (\rho_0, \infty)$ for some $ \rho_0 > 0$.
$\{ \Sigma_\rho \}$ is called {\em nearly round} if
there exists a constant $ C>0$ such that, for all $ \rho$,  the following are satisfied:

\begin{enumerate}
\item[(i)] $ C^{-1} \rho \le | x | \le C \rho $ for all $ x \in \Sigma_\rho$;

\item[(ii)] $  | \ringA | + \rho | \nabla \ringA | \le C \rho^{-1- \tau} $;

\item[(iii)]  $ | \S_\rho | \le C\rho^{n-1}$;
  \item[(iv)] $\text{Diam}(\S_\rho)\le C\rho$.
\end{enumerate}

\end{definition}

Next, we  give a series of lemmas for a family of nearly round hypersurfaces $\{\Sigma_\rho \}$ in an asymptotically flat $ (M^n , g)$.

\begin{lma} \label{lma-grad-H-A-decay}
There is  a constant $C>0$ independent on $\rho$ such that
$$ | \nabla A | \le C \rho^{-2 - \tau} , \ | \nabla H | \le C \rho^{-2 - \tau} . $$
\end{lma}

\begin{proof}
By \cite[Lemma 2.2]{Huisken1986}, given any constant $a>0$,
\be \label{eq-A-H-Ricci}
\begin{split}
|\nabla A|^2 \ge & \ \lf(\frac{3}{n+1}-a\ri)|\nabla H|^2 \\
& \ -\frac{2}{n+1}\lf(\frac{2}{n+1}a^{-1}-\frac{n-1}{n-2}\ri)|w|^2 ,
\end{split}
\ee
where $w$ is the projection of $\Ric(\nu,\cdot)$ on $\S_\rho$ and
 $ \Ric (\cdot, \cdot)$ is  the Ricci curvature of $(M^n, g)$.
On the other hand, given any constant $ \epsilon  > 0 $,
\be \label{eq-H-A-TA}
\frac{1}{n-1}|\nabla H|^2
= |\nabla (A-\ringA)|^2\ge (1-\e)|\nabla A|^2-C(\e)|\nabla\ringA|^2
\ee
for some constant $ C(\epsilon) > 0 $ depending only on $ \epsilon $.
Since $ n \ge 3$, we may choose $ a $ and $ \epsilon $ such that
$
\lf(\frac{3}{n+1}-a\ri)(n-1)(1-\epsilon )>1
$.
Thus, it follows from \eqref{eq-A-H-Ricci} and  \eqref{eq-H-A-TA} that
\be \label{eq-est-A}
|\nabla A|\le C\rho^{-2-\tau},
\ee
where we also used \eqref{eq-AF} and (i), (ii) in Definition \ref{def-nearly-round}.
The fact  $ | \nabla H | \le C \rho^{-2 - \tau} $ now follows from \eqref{eq-A-H-Ricci} and \eqref{eq-est-A}.
\end{proof}

In what follows, we  let $ \hat g $ be the background
 Euclidean metric on the open set of $M^n$ that is diffeomorphic to $ \R^n \setminus \{ | x | \le r_0 \}$.
A symbol with  ``$\ \hat \ \ $"    means  the  corresponding quantity is  computed with respect to $ \hat g$.

\begin{lma}  \label{lma-est-A-hat-A}
As $ \rho \rightarrow \infty$,
$$   d \sigma = ( 1 + O (\rho^{-\tau}) ) d \hat \sigma , \ \nu - \hat \nu = O ( \rho^{-\tau}), $$
$$  | A - \hat A | = O ( \rho^{-\tau} ) | \hat A | + O ( \rho^{-1 -\tau} ) .$$
\end{lma}

\begin{proof}
At any  $ p \in \Sigma_\rho$, let $ \{ e_1, \ldots, e_{n-1} \}$ be a $\hat g$-orthonormal frame in $ T_p \Sigma_\rho$.
Let $ e_n = \hat \nu $, then $ \{ e_i \ | \ i =1, \ldots, n \}$ form  a $\hat g$-orthonormal frame in $ T_p M^n$.
Hence, if   $ e_i = a_i^j \p_{x_j} $, then $ (a_i^j)_{n \times n}  $ is an orthogonal matrix.
Thus,
\be \label{eq-g-e-ij}
g (e_i, e_j) =    a_i^k a_j^l g( \p_{x_k}, \p_{x_l} )
=  \delta_{ij} + O (\rho^{-\tau} ),
\ee
which implies $ d \sigma = ( 1 + O ( \rho^{-\tau} ) ) d \hat \sigma$.
Now assume $ \nu = X^i e_i$. By \eqref{eq-g-e-ij} and the  fact $ g (\nu, \nu) = 1 $,
\be
\begin{split}
1 = & \ X^i X^j \lf[  \delta_{ij} + O ( \rho^{-\tau} ) \ri],
\end{split}
\ee
which implies  $ X^i = O (1) $ and hence
\be \label{eq-Xi-square}
\sum_i (X^i)^2 = 1 + O ( \rho^{- \tau} ).
\ee
Similarly,  for $ 1 \le \alpha \le  n-1$, the fact $ g(\nu, e_\alpha) = 0 $ and \eqref{eq-g-e-ij}
 imply
\be \label{eq-X-alpha}
0 =  X^\alpha  + O ( \rho^{- \tau} )  .
\ee
Therefore,  by \eqref{eq-Xi-square} and \eqref{eq-X-alpha},
\bee
\nu - \hat \nu = X^\alpha e_\alpha  + ( X^n - 1) e_n = O (\rho^{-\tau}).
\eee

To  compare $ A $ and $ \hat A$ at $p$,  we let $ W_\rho   $ be a small open set  in $\Sigma_\rho$ containing $p$.
  Let $ U_\rho $ be a $\hat g$-Gaussian tubular neighborhood of  $ W_\rho $ in $M^n$
and let  $\{  u_1, \ldots, u_{n-1} , t \}  $  be  local coordinates  on $ U_\rho $ such that, at $ t=0$,
$  \{ {u_\alpha} \ | \   \alpha =1, \ldots,   n-1 \} $  are local coordinates on  $ W_\rho $
satisfying  $ \hat \nabla_{\p_{u_\alpha} } \p_{u_\beta} = 0 $ at $p$.
In $ U_\rho $,  suppose  $ \p_{u_\alpha} = X^i_\alpha  \p_{x_i} $.
Then
$$
\nabla_{\p_{u_\alpha} } \p_{ u_\beta}   =  X^i_\alpha  \lf( X^j_\beta \nabla_{\p_{x_i}} \p_{x_j}
+   \frac{\p X^j_\beta }{ \p {x_i} } \p_{x_j} \ri) ,
$$
$$
\hat \nabla_{\p_{u_\alpha} } \p_{ u_\beta}   =  X^i_\alpha  \lf( X^j_\beta \hat \nabla_{\p_{x_i}} \p_{x_j}
+   \frac{\p X^j_\beta }{ \p {x_i} } \p_{x_j} \ri)  ,
$$
which shows, at $p$,
\bee
\nabla_{\p_{u_\alpha} } \p_{ u_\beta} = \hat \nabla_{\p_{u_\alpha} } \p_{ u_\beta} + O ( \rho^{-1 - \tau} ) .
\eee
Therefore,
\be \label{eq-A-hatA}
\begin{split}
g(\nabla_{\p_{u_\alpha} } \p_{u_\beta} , \nu )
= & \  g(\hat \nabla_{\p_{u_\alpha } } \p_{u_\beta} , \nu )  + O ( \rho^{-1 - \tau} ) \\
= & \  g(\hat \nabla_{\p_{u_\alpha}  } \p_{u_\beta} , \hat \nu )  + O(\rho^{-\tau} ) | \hat A | + O ( \rho^{-1 - \tau} ) \\
= & \ \hat g(\hat \nabla_{\p_{u_\alpha}  } \p_{u_\beta} , \hat \nu ) + O(\rho^{-\tau} ) | \hat A | + O ( \rho^{-1 - \tau} ) .
\end{split}
\ee
From this we conclude
\be  \label{eq-norm-A-hatA}
 | A - \hat A | = O ( \rho^{-\tau} ) | \hat A | + O ( \rho^{-1 -\tau} ) .
\ee
\end{proof}

\begin{lma}  \label{lma-H-kappa}
For each large $\rho$, there exists $r (\rho)>0$ with  $ r (\rho) \sim \rho$, meaning
$ C^{-1} \rho < r(\rho) < C \rho $ for some constant $ C$ independent on $\rho$, such that
\bee
H = \frac{n-1}{r (\rho) } + O(  \rho^{-1- \tau} ) .
\eee
Consequently, $ |A| = O ( \rho^{-1} ) $, $ |\hat A| = O(\rho^{-1}) $ and
$$ \kappa_\alpha =\frac1{r(\rho)}+O(\rho^{-1-\tau}),  \ \hat\kappa_\alpha =\frac1{r(\rho)}+O(\rho^{-1-\tau}),
 $$
for all $ \alpha =1, \ldots, n-1$, where $ \kappa_\alpha$ and $ \hat \kappa_\alpha$ denote the principal
curvature of $\Sigma_\rho$ with respect to $g$ and $\hat g$ respectively.
\end{lma}

\begin{proof}
As $ \Sigma_\rho$ is homeomorphic to a sphere, for each $\rho$ large,
there is an $r_1 = r_1 (\rho) > 0$, such that $ S_{r_1} = \{ | x | = r_1 \}$ is the largest
coordinate sphere inside $ \Sigma_\rho$. Let $ x_1 $ be a point where $ S_{r_1}$ touches $ \Sigma_\rho$.
It follows from the maximum principle that
\be  \label{eq-H-comparison}
H (x_1) \le \bar{H} (x_1)
\ee
where $ \bar{H}$ is the mean curvature of $ S_{r_1}$ in $(M^n, g)$ with respect to the outward normal.
Direct calculation gives
\be \label{eq-H-sphere}
\bar{H} ( x_1 ) = \frac{n-1}{r_1} + O ( |r_1|^{ - 1 - \tau} )
\ee
(cf. Lemma 2.1 in \cite{FanShiTam} when $n=3$).
By Lemma \ref{lma-grad-H-A-decay} and (iv) in Definition \ref{def-nearly-round},   $H$ satisfies
\be \label{eq-H-two-points}
| H  - H(x_1) | \le C \rho^{-1- \tau} .
\ee
Hence, it follows from  \eqref{eq-H-comparison} -- \eqref{eq-H-two-points} that
\bee
H  \le  \frac{n-1}{r_1} + O ( \rho^{ - 1 - \tau} ) .
\eee
Similarly, if  $ r_2 > 0 $ is  the radius of the smallest coordinate sphere that lies outside of $\Sigma_\rho$, then
\bee
H  \ge   \frac{n-1}{r_2} + O ( \rho^{ - 1 - \tau} ) .
\eee
For any fixed point $x_* \in \Sigma_\rho$, these imply
\be \label{eq-rrho-sim}
  \frac{n-1}{r_2} + O ( \rho^{ - 1 - \tau} ) \le H (x_*) \le   \frac{n-1}{r_1} + O ( \rho^{ - 1 - \tau} ) .
\ee
By the same reasoning leading to \eqref{eq-H-two-points},
\bee
| H  - H(x_*) | \le C \rho^{-1- \tau} .
\eee
  Hence, choosing $r(\rho) = \frac{n-1}{H(x_*)}$,
we have
\bee
H = \frac{n-1}{r (\rho) } + O(  \rho^{-1- \tau} ),
\eee
where $ r (\rho) \sim \rho $ by \eqref{eq-rrho-sim}.  As a result, $A$ satisfies
\bee
A = \ringA +  \lf[ \frac{1}{r (\rho)}  + O(  \rho^{-1- \tau} ) \ri] \sigma,
\eee
which implies $ |A| = O (\rho^{-1} )$ and
\bee
\kappa_\alpha =  \frac{1}{r (\rho)}  + O(  \rho^{-1- \tau} ), \ \forall \ \alpha
\eee
by (ii) in Definition \ref{def-nearly-round}.
The claim on $ | \hat A |$ and $\hat \kappa_\alpha$ now follows from Lemma \ref{lma-est-A-hat-A}.
\end{proof}

In the rest of this section, $r (\rho)$ will always denote     the function of $\rho$  given in Lemma \ref{lma-H-kappa}.

\begin{lma} \label{lma-sectional}
Let $K_1(\rho)$ and $K_2(\rho)$ be the minimum and maximum of the sectional curvature of $\S_\rho$ respectively. Then
$$
K_1 (\rho) = \frac1{r^2(\rho)} + O(\rho^{-2-\tau}) , \ \    K_2(\rho) =  \frac1{r^2(\rho)}+O(\rho^{-2-\tau}).
$$
\end{lma}

\begin{proof}
This follows from the Gauss equation, Lemma \ref{lma-H-kappa} and the fact that
 the sectional curvature of $(M^n, g)$ at $ x \in \Sigma_\rho $  decays like $O(\rho^{-2-\tau})$.
 \end{proof}

\begin{lma} \label{lma-sphere-center}
Let $ X = (x^1, \ldots, x^{n})$ be  the  position vector given by  the coordinates $\{ x^i \}$ near infinity.
For each large $\rho$,
there exists  a vector  $\mathbf{a}(\rho) \in \R^n$ such that  $|\mathbf{a}(\rho)| = O ( \rho) $ and, on $ \Sigma_\rho$,
$$|X-\mathbf{a}(\rho) -r(\rho)\hat \nu|=O(\rho^{1-\tau}) . $$
\end{lma}

\begin{proof}
 At any $p\in \S_\rho$, let $ \{ e_\alpha \}  \subset T_p \Sigma_\rho $ be a $\hat g$-orthonormal frame that diagonalizes $\hat A$, i.e.
 $\hat A( e_\alpha , e_\beta ) = \hat \kappa_\alpha \delta_{\alpha \beta} $, then
\bee \label{eq-grad-vector-function}
|\hat\nabla_{ e_\alpha }( X- r(\rho)\hat \nu)|= | e_\alpha -r(\rho)\hat\kappa_\alpha  e_\alpha |  =O(\rho^{ -\tau})
\eee
by Lemma \ref{lma-H-kappa}.
This and (iv) in Definition \ref{def-nearly-round}  imply
\bee
|(X(p)-r(\rho)\hat \nu(p))-(X(p_0)-r(\rho)\hat \nu(p_0))|=O(\rho^{1-\tau})
\eee
for all $p\in \S_\rho$ and a  fixed $p_0 \in \Sigma_\rho$.
Let  $\mathbf{a}(\rho)= X(p_0)-r(\rho)\hat \nu(p_0)$,  the lemma follows.
\end{proof}

\begin{lma} \label{lma-area-and-r}
For large $\rho$, the volume $ | \Sigma_\rho |$  satisfies
\bee
  \lf( \frac{ | \Sigma_\rho |  }{ \omega_{n-1} } \ri)^\frac{1}{n-1} =  r(\rho) ( 1 + O ( \rho^{-\tau} ) ).
\eee
\end{lma}

\begin{proof} Let $ | \Sigma |_{\hat g} $ be  the volume of $\Sigma_\rho$ with respect to the metric induced
from $\hat g$.
Since $ d \sigma = ( 1 + O (\rho^{-\tau} ) ) d \hat \sigma $  by Lemma \ref{lma-est-A-hat-A},
we have
\be \label{eq-area-ratio}
\frac{ | \S_\rho | }{ | \S_\rho |_{\hat g} } = 1 + O (\rho^{-\tau} ) .
\ee
Now let $ \mathbf{a}(\rho)$ be the vector given in  Lemma 2.5, then
 $$  |X-\mathbf{a}(\rho)| =  r(\rho)\lf(1+ O( \rho^{-\tau} ) \ri).
 $$
Hence,
$\Sigma_\rho$ contains a Euclidean sphere ${S}_1$, centered at $ \mathbf{a}(\rho)$,  of radius $r_1$,
and is contained in another  Euclidean sphere $S_2$,  centered at $ \mathbf{a}(\rho)$,  of radius $r_2$,
such that both $r_1$ and $ r_2 $ satisfy
\be \label{eq-r1-and-r2}
  r_1 =  r(\rho)\lf(1+ O( \rho^{-\tau} ) \ri), \ r_2 =  r(\rho)\lf(1+ O( \rho^{-\tau} ) \ri) .
\ee
Let $\eta$ be the Euclidean distance function  from $\mathbf{a}(\rho)$. Integrating $\Delta_{\hat g} \eta$
over the domain bounded between $\Sigma_\rho$ and $S_1$, where $ \Delta_{\hat g}$
denotes the Euclidean Laplacian,   we conclude
 \bee
 |\Sigma_\rho|_{\hat g}\ge  | S_1|_{\hat g}
 =\omega_{n-1} r_1^{n-1} .
 \eee
 Hence,
 \be\label{e-AF-Hawking-3}
 \lf(\frac{|\Sigma_\rho|_{\hat g}}{\omega_{n-1}}\ri)^{\frac1{n-1}}\ge r_1 .
 \ee
 Next, let $\zeta$ be the Euclidean distance function from $\Sigma_\rho$, defined on the exterior of $ \Sigma_\rho $.
 For large $\rho$,  $\Sigma_\rho$ is strictly convex in $ \R^n$ by Lemma \ref{lma-sectional}.
 Therefore,  $\zeta$ is smooth   and $ \Delta_{\hat g} \zeta$ equals
 the mean curvature of the level set of $\zeta$ in  $ \R^n$  and hence is positive.
Integrating $\Delta_{\hat g}\zeta$ over the domain bounded by $S_{r_2}$ and  $\Sigma_\rho$,
 arguing  as before, we conclude
 \be\label{e-AF-Hawking-4}
 \lf(\frac{|\Sigma_\rho|_{\hat g}}{\omega_{n-1}}\ri)^{\frac1{n-1}}\le r_2 .
 \ee
Hence, by \eqref{eq-r1-and-r2} -- \eqref{e-AF-Hawking-4},
\be \label{eq-inter-between}
\lf(\frac{|\Sigma_\rho|_{\hat g}}{\omega_{n-1}}\ri)^{\frac1{n-1}} =  r(\rho)\lf(1+ O( \rho^{-\tau} ) \ri) .
\ee
The lemma now follows from \eqref{eq-area-ratio} and \eqref{eq-inter-between}.
\end{proof}

When $ n =3$,  Lemma \ref{lma-sectional} implies, for large $\rho$,
$ \Sigma_\rho$ can be isometrically embedded in $ \R^3$
as a convex surface (cf. \cite{Nirenberg}). In \cite{ShiWangWu09}, Shi, Wang and Wu proved
that the principal curvatures $\ko_1$, $\ko_2$ of such an embedding satisfy
\be \label{eq-SWW}
{\ko_1} = \frac{1}{r (\rho) } + O (\rho^{-1-\tau} )  \ \ \mathrm{and} \ \ {\ko_2} = \frac{1}{r (\rho)} + O (\rho^{-1-\tau} )
\ee
(see \cite[Theorem 4]{ShiWangWu09} and its proof).
In the next lemma, assuming the isometric embedding exists, we show that such estimates
 hold in higher dimensions .

\begin{lma} \label{lma-embedding-AF}
Suppose $ n \ge 4$ and assume
$\S_\rho$ can be isometrically embedded in $\R^n$ when  $\rho$ is sufficiently large.
Let $\S_\rho^{(0)}$ be the image of  the embedding.
Let $\ko_\alpha$, $\alpha =1, \ldots, n-1$, be the principal curvatures of $ \S_\rho^{(0)}$ in $ \R^n$.
Then
$$\ko_\alpha =\frac1{r(\rho)}+O(\rho^{-1-\tau}). $$
\end{lma}

\begin{proof} Let $H_0$ be the mean curvature of $\S_\rho^{(0)}$ in $ \R^n$ with respect to the outward normal.
At a point $p_0 \in \S_\rho^{(0)}$, let $\{ e_\alpha^{(0)} \}$ be an orthonormal basis in $ T_{p_0} \S_\rho^{(0)}$ such that
$ e_\alpha^{(0)}$ points to  the principal direction of $\S_\rho^{(0)}$.
Let $R^{\rho}_{\alpha \beta \gamma \delta}$, $R^{\rho}_{\alpha \beta }$
and $S^\rho$  be the intrinsic curvature tensor, the Ricci tensor and the scalar curvature of $\S_\rho^{(0)}$ respectively.
For convenience,  let $r =r(\rho)$.
By Lemma \ref{lma-sectional},
\be\label{e-estimate-1}
 R^{\rho}_{\alpha \alpha }=\frac{n-2}{r^2}+O(\rho^{-2-\tau}), \ S^\rho=\frac{(n-1)(n-2)}{r^2}+O(\rho^{-2-\tau}).
 \ee
 Moreover, by the Gauss equation,
\be \label{eq-gauss-1}
\ko_\alpha \ko_\beta =R^{\rho}_{\alpha \beta \alpha \beta}=\frac1{r^2}+O(\rho^{-2-\tau}) , \ \forall \ \alpha \neq \beta.
\ee
 In particular, for large $\rho$, this implies $ \S_\rho^{(0)}$ is convex, i.e.
$ \ko_\alpha > 0 $, $\forall \ \alpha$. As a result, $H_0 > 0$.

Summing over indices in \eqref{eq-gauss-1} repeatedly, we have
 \be \label{eq-gauss-2}
 R^{\rho}_{\alpha \alpha}=H_0\ko_\alpha -(\ko_\alpha)^2
 \ee
and
\be \label{eq-gauss-3}
S^\rho = H_0^2 - | \Pi_0 |^2
\ee
 where $ \Pi_0 $ is  the second fundamental form of $\S_\rho^{(0)}$.
 By \eqref{eq-gauss-3} and  \eqref{e-estimate-1},
\bee
H_0^2
\ge \frac1{n-1}H_0^2+\frac{(n-1)(n-2)}{r^2}+O(\rho^{-2-\tau}),
\eee
which implies
\be\label{e-estimate-2}
H_0\ge \frac{n-1}r+O(\rho^{-1-\tau}).
\ee
 because $H_0>0$. On the other hand,
by \eqref{eq-gauss-2},
$$
\ko_\alpha =\frac12\lf[H_0\pm (H_0^2-4R^{\rho}_{\alpha \alpha})^\frac12\ri].
$$
 We claim that,
\be \label{eq-kappa-root}
\ko_\alpha =\frac12\lf[H_0-(H_0^2-4R^{\rho}_{\alpha \alpha})^\frac12\ri], \ \forall \ \alpha.
\ee
Suppose not,  without losing generality, we may assume
\bee
\kappa_1^{(0)}=\frac12\lf[H_0 + (H_0^2 - 4R^{\rho}_{11 })^\frac12\ri] .
\eee
Then $ \kappa_1^{(0)} \ge \frac12 H_0 $. This  implies, $ \forall \ \alpha > 1 $,
\be\label{e-estimate-3}
  \kappa_\alpha^{(0)}=  \frac12\frac{ 4R^{\rho}_{\alpha \alpha }}{H_0+  (H_0^2 - 4R^{\rho}_{\alpha \alpha })^\frac12} ,
   \ee
for  otherwise we would have  $\ko_1+ \sum_{\alpha > 1} \ko_\beta > H_0$ since $ n \ge 4$ and $ \ko_\alpha > 0$, which is a contradiction.
By \eqref{e-estimate-1} and \eqref{e-estimate-2},
\be \label{e-estimate-3-5}
\begin{split}
H_0^2-4R^{\rho}_{\alpha \alpha }\ge & \ \frac{(n-1)^2}{r^2}-\frac{4(n-2)}{r^2}+O(\rho^{-2-\tau}) \\
= & \ \frac{(n-3)^2}{r^2}+O(\rho^{-2-\tau}) .
\end{split}
\ee
Hence, it follows from \eqref{e-estimate-2}, \eqref{e-estimate-3} and \eqref{e-estimate-3-5} that, $ \forall \ \alpha > 1 $,
  \be \label{e-estimate-5}
  \begin{split}
  \kappa_\alpha^{(0)}
    \le & \ \frac{\frac{2(n-2)}{r^2 }+O(\rho^{-2-\tau})}{\frac{2(n-2)}{r}+O(\rho^{-1-\tau}) }\\
  =& \ \frac1r+O(\rho^{-1-\tau}).
\end{split}
  \ee
On the other hand,
for any  $ \beta $, $ \gamma$ such that    $ \beta > 1$,   $ \gamma >1$ and $ \beta \neq \gamma$,
  \be \label{eq-gauss-1-1}
  \kappa_\beta^{(0)}\kappa_\gamma^{(0)} =\frac1{r^2}+O(\rho^{-2-\tau})
  \ee
  by \eqref{eq-gauss-1}. (As $ n \ge 4$, such indices $\beta $ and $ \gamma$ exist.)
  Hence,  \eqref{e-estimate-5} and \eqref{eq-gauss-1-1} show
  \bee
  \frac1{r^2}+O(\rho^{-2-\tau})\le \kappa_\beta^{(0)}\lf(\frac1r+O(\rho^{-1-\tau})\ri) ,
  \eee
which gives
  \be \label{e-estimate-6}
  \begin{split}
  \kappa_\beta^{(0)}
  \ge & \ \frac1r+O(\rho^{-1-\tau}).
\end{split}
  \ee
Therefore,  by \eqref{e-estimate-5} and \eqref{e-estimate-6},
  \be\label{e-estimate-7}
  \kappa_\alpha ^{(0)}=\frac1r+O(\rho^{-1-\tau}),
  \ee
  for all  $ \alpha >  1$,
  But then
  \bee
  \frac1{r^2}+O(\rho^{-2-\tau}) =  \kappa_1^{(0)}\kappa_2^{(0)}
  \ge  \frac12 H_0 \kappa_2^{(0)}
  \ge  \frac{n-1}{2r^2}+O(\rho^{-2-\tau}),
  \eee
  which is impossible since  $n \ge 4$.
  Therefore,
  the claim
  \eqref{eq-kappa-root} holds.
  Now \eqref{e-estimate-3} is valid  for all $\alpha \ge 1$. Repeating the argument leading to \eqref{e-estimate-7}, we conclude
  \eqref{e-estimate-7} holds for all $ \alpha \ge 1$.  This completes  the proof.
\end{proof}

We now recall the statement of Theorem \ref{t-intro-1} and give its proof.


\begin{thm}  \label{thm-section-AF}
Let $ \{ \Sigma_\rho \}$ be a family of nearly round hypersurfaces in an asymptotically flat manifold $(M^n, g)$
of dimension $ n \ge 3 $. Then
\be \label{eq-thm-S-H}
 \lim_{\rho \to \infty}  \frac{c_n}{2}  \lf(\frac{|\Sigma_\rho| }{\omega_{n-1}}\ri)^{\frac1{n-1}}
 \int_{\S_\rho}   \lf(    S^\rho - \frac{n-2}{n-1} H^2   \ri) d\sigma =  \m  .
 \ee
If in addition  $ \Sigma_\rho $ can be isometrically embedded in $ \R^n$
when $\rho$ is   sufficiently large, which is automatically satisfied if  $n=3$,
then
\be \label{eq-thm-BY}
  \lim_{\rho \to \infty}  2 b_n \int_{\S_\rho}  (H_0 - H )d\sigma  =  \m
\ee
and
\be \label{eq-thm-sigma-2}
  \lim_{\rho \to \infty}  c_n    \lf(\frac{|\Sigma_\rho| }{\omega_{n-1}}\ri)^{\frac1{n-1}}    \int_{\S_\rho} \sum_{\alpha < \beta} \lf( \ko_\alpha \ko_\beta  - \kappa_\alpha \kappa_\beta \ri) d \sigma
=  \m   .
\ee
Here  $ | \Sigma_\rho |$, $ S^\rho$  are  the volume, the scalar curvature  of $\Sigma_\rho$, respectively;
 $H$, $ \{ \kappa_\alpha\}$  are the mean curvature, the principal curvatures  of $\Sigma_\rho$ in $ (M^n ,g )$,  respectively;  $H_0$,  $ \{ \ko_\alpha\}$  are the mean curvature,  the principal curvatures
  of the isometric embedding of $\Sigma_\rho$ in $ \R^n$, respectively;
 and $\m $ is the total mass of $(M^n, g)$.
\end{thm}

\begin{proof}
For simplicity,  denote $G^g_0$ by $G$ and denote $ r(\rho)$ by $ r$.
On $\Sigma_\rho$,  by Lemma \ref{lma-sphere-center} and Lemma \ref{lma-est-A-hat-A}, we have
\be\label{e-BY-1}
\begin{split}
  G (X,\nu)
= & \   G(r  \hat\nu,\nu)+ G (\mathbf{a}(\rho) ,\nu)+O(\rho^{-1-2\tau}) \\
= & \ r  G (\nu,\nu)+ G (\mathbf{a}(\rho) ,\nu)+O(\rho^{-1-2\tau})  .
\end{split}
\ee

We first    estimate   $ \int_{\Sigma_\rho} G (\mathbf{a}(\rho) ,\nu) d \sigma$.
To do so, we note    the following asymptotic formulae    of $\Ric (g) $
and $ \mathcal{S}_g $ (see (2.2) and (2.6) in  \cite{MiaoTam2015}):
   \be \label{eq-Ricci}
   2\Ric(g)_{ij}=g_{ki,kj}+g_{kj,ki}-g_{ij,kk}-g_{kk,ij}+O( \rho^{-2-2\tau})  ,
   \ee
   \be \label{eq-scalar}
   \mathcal{S}_g =g_{ik,ik}-g_{kk,ii}+O(\rho^{-2-2\tau}).
   \ee
   Here and below, summation is performed over any pair of repeated indices.

Write   $\mathbf{a}(\rho)=(a^1,\dots,a^n)$, which is a constant vector  when $\rho $ is fixed.
By  Lemma \ref{lma-est-A-hat-A},  Lemma \ref{lma-sphere-center}, \eqref{eq-Ricci} and \eqref{eq-scalar}, we have
   \be \label{e-2nd-term-1}
   \begin{split}
  2 G (\mathbf{a}(\rho),\nu)d\sigma
=& \ \bigg[\lf(g_{ki,kj}+g_{kj,ki}-g_{ij,kk}-g_{kk,ij}\ri)a^i\hat\nu^j
   \\
   &- \lf(g_{ik,ik}-g_{kk,ii}\ri) a^m\hat\nu^m+O(\rho^{-1-2\tau})\bigg]d\hat\sigma .
   \end{split}
   \ee
In what follows, we arbitrarily extend $g_{ij}$ as a smooth function  on the whole Euclidean
space $ \R^n$ such that $ g_{ij}$ remains unchanged on $ \R^n \setminus \{ | x | \le 2 r_0 \}$
(see Definition \ref{def-AF}).
We identify $ \Sigma_\rho $ with its image under the diffeomorphism that defines the coordinates $\{ x_i \}$.
Let $ D_\rho \subset \R^n$ be the bounded region enclosed by $ \Sigma_\rho$ and let $ d \hat V$ be the
Euclidean volume form on $ \R^n$. Then
   \be \label{e-2nd-term-2}
   \begin{split}
  & \ \int_{\S_\rho}\bigg[\lf(g_{ki,kj}+g_{kj,ki}-g_{ij,kk}-g_{kk,ij}\ri)a^i\hat\nu^j ]d\hat\sigma \\
   =& \ \int_{D_\rho} \lf(g_{ki,kjj}+g_{kj,kij}-g_{ij,kkj}-g_{kk,ijj}\ri)a^i d \hat V \\
   =& \ \int_{D_\rho} \lf( g_{kj,kij} -g_{kk,ijj}\ri)a^i d\hat V  \\
   =& \ \int_{\S_\rho} \lf( g_{kj,kj} -g_{kk,jj}\ri)a^i \hat \nu^i d \hat \sigma .
   \end{split}
   \ee
It  follows from  \eqref{e-2nd-term-1}, \eqref{e-2nd-term-2} and (iii) in Definition \ref{def-nearly-round}
that
   \be \label{e-2nd-term}
   2 \int_{\S_\rho} G (\mathbf{a}(\rho) ,\nu)d\sigma
   =  o(1) .
   \ee

 By  \eqref{e-BY-1} and \eqref{e-2nd-term},
 we conclude
 \be \label{eq-Gmass}
 \int_{\S_\rho} G(X, \nu) d \sigma  = r \int_{\S_\rho} G(\nu, \nu) d \sigma   + o(1) .
 \ee
   The Gauss equation implies
\be  \label{e-gauss-1-pf}
\begin{split}
G (\nu,\nu) =  & \ \frac12 \lf( H^2 - | A|^2 - S^\rho \ri) \\
 = & \  \sum_{\alpha < \beta} ( \kappa_\alpha \kappa_\beta - \ko_\alpha \ko_\beta ).
\end{split}
\ee
By Lemma \ref{lma-H-kappa} and Lemma  \ref{lma-embedding-AF}, when $n \ge 4$,
\be  \label{eq-sigma-2-difference}
 \begin{split}
&  \  \kappa_\alpha \kappa_\beta - \kappa_\alpha^{(0)}\kappa_\beta ^{(0)} \\
 = & \ \lf(\kappa_\alpha -\kappa_\alpha ^{(0)}\ri)\kappa_\beta  + \lf(\kappa_\beta -\kappa_\beta ^{(0)}\ri)\kappa_\alpha^{(0)}\\
 =&\lf(\kappa_\alpha - \kappa_\alpha^{(0)}\ri)\lf(\kappa_\beta -\frac1r \ri) + \lf(\kappa_\beta -\kappa_\beta^{(0)}\ri)\lf(\kappa_\alpha ^{(0)}-\frac1r\ri)\\
 &+\frac1r\lf(\kappa_\alpha +\kappa_\beta -\ko_\alpha -\ko_\beta \ri)\\
 =&\frac1r\lf(\kappa_\alpha +\kappa_\beta -\ko_\alpha -\ko_\beta \ri)+O(\rho^{-2- 2 \tau}).
 \end{split}
 \ee
 When $ n =3$, \eqref{eq-sigma-2-difference} is also true   by Lemma \ref{lma-H-kappa}
 and  the estimate \eqref{eq-SWW}.
Therefore, by \eqref{e-gauss-1-pf}, \eqref{eq-sigma-2-difference} and
 the fact $ | \ringA | = O(\rho^{-1-\tau})$,
 we  have
 \be\label{e-BY-2}
\begin{split}
 r\int_{\S_\rho}  G (\nu,\nu)d\sigma
=  (n-2) \int_{\S_\rho}  (H - H_0 )d\sigma + o(1)
\end{split}
\ee
and
 \be \label{e-BY-2-5}
 r\int_{\S_\rho}  G(\nu,\nu)d\sigma
=  r\int_{\S_\rho} \frac12  \lf( \frac{n-2}{n-1} H^2  - S^\rho \ri) d\sigma + o (1) .
\ee
The theorem  now  follows readily   from \eqref{eq-Gmass}, \eqref{e-gauss-1-pf},
\eqref{e-BY-2},   \eqref{e-BY-2-5}, Lemma \ref{lma-area-and-r}
and Theorem \ref{t-mass}(i) with $\{ S_r \}$ replaced by  $\{\Sigma_\rho\}$
 (see \cite[Theorem 2.1]{MiaoTam2015}).
\end{proof}

\section{Limits of quasi-local mass integrals in AH  manifolds} \label{section-ah}

Given  a compact manifold $(\Omega^n,g)$ with boundary $\Sigma $  and assuming  $\Sigma $  can be isometrically embedded in $\H^n$,
 motivated by the work \cite{WangYau2007,ShiTam2007},
 we are interested in  the
vector-valued integral
\be\label{e-ql-AH}
\mathbf{m}(\Omega,\Sigma )=\int_{\Sigma }(H_0-H) \,  \mathbf{x} \, d\sigma ,
\ee
where $H$ is the mean curvature of $\Sigma $
 in $(\Omega^n, g)$,
$H_0$ is the mean curvature of the embedding of $\Sigma $  in $\H^n$, and $\mathbf{x}$ is
the position vector of points in  $ \H^n \subset \R^{n,1}$,
where
\be\label{e-Hn}
\H^n=\{(x^0,x^1,\dots,x^n)\in \R^{n,1}|\ (x^0)^2-\sum_{i=1}^n(x^i)^2=1, x^0>0\} .
\ee
When $n=3$, if $\Sigma $ is homeomorphic to $ \mathbb{S}^2$ with Gaussian curvature larger than $-1$,
then $\Sigma $ can be isometrically embedded in $\H^3$ and the embedding is unique up to isometries of $\H^3$ (cf. \cite{P}).
Note that the integral  in \eqref{e-ql-AH} depends also on the embedding of $\Sigma$.

In this section, we are interested in analyzing the asymptotic behavior of
$\mathbf{m}(\Omega_\rho, \Sigma_\rho)$, together with other related geometric quantities,
as $\rho \rightarrow \infty$. Here
 $ \Omega_\rho $ is  the bounded domain enclosed by $ \Sigma_\rho $
 on an asymptotically hyperbolic manifold $(M^n, g)$ defined in Definition \ref{d-AH}.
We continue to use   $H$,  $ {S}^\rho $ to   denote the mean curvature,
the scalar curvature respectively of the hypersurfaces $\Sigma_\rho$, which
are  the geodesic spheres  in  the background $\H^n$.

Recall that
the ball model $(B^n,g_0)$ of $\H^n$ is given by
$$
B^n=\{x\in \R^n|\ |x|<1\},\ \ g_0=\frac{4}{(1-|x|^2)^2} g_e ,
$$
where $g_e$ is the standard Euclidean metric.

\begin{thm}\label{t-AH-1}
Let  $(M^n,g)$ be an asymptotically hyperbolic manifold.
  Let $\Sigma_\rho$ be  the hypersurface in $M^n$ corresponding to the
  geodesic sphere of radius $\rho$  in  $\H^n$ as specified  in Definition \ref{d-AH}.
  Let $\{ V^{(i)} \}_{0\le i\le n-1} $ be the functions given in \eqref{e-def-Vi}.
  Let $ | \Sigma_\rho |$ denote the volume of $\Sigma_\rho$ and $ d \sigma_\rho$ be the volume element on $\Sigma_\rho$.
  Then the following are true

\begin{enumerate}
  \item [(I)]
   \begin{itemize}
   \item[(a)]   \bee \label{eq-MV0-Hmass}
\begin{split}
& \ \mathbf{M}(g) (V^{ (0)}) \\
=  & \   \lim_{ \rho \rightarrow \infty}   \frac{c_n}{2}    \lf( \frac{ | \Sigma_\rho |  }{ \omega_{n-1} } \ri)^\frac{1}{n-1}
  \int_{\Sigma_\rho}  \lf[   {S}^\rho  - \frac{n-2}{n-1} H^2  + (n-1)(n-2)\ri]  d\sigma_\rho     .
\end{split}
\eee
    \item[(b)]   If  $\tau>n-1$, then for $1\le i\le n$,
    \bee
\begin{split}
& \ \mathbf{M}(g) (V^{ (i)}) \\
=  & \   \lim_{ \rho \rightarrow \infty}  \frac{c_n}{2}    \lf( \frac{ | \Sigma_\rho |  }{ \omega_{n-1} } \ri)^\frac{1}{n-1}
    \int_{\Sigma_\rho} \frac{x^i}{|x|} \lf[  {S}^\rho -  \frac{n-2}{n-1} H^2 + (n-1)(n-2) \ri]  d\sigma_\rho     ,
\end{split}
    \eee
where $\{ x^i \} $ are the coordinate functions in the ball model  $(B^n, g_0)$ of $\H^n$.
  \end{itemize}

  \item [(II)] Let $ g_\rho$ denote  the  metric on $ \Sigma_\rho$ induced from $g$.
Suppose    $(\Sigma_\rho,g_\rho)$ can be isometrically embedded in $\H^n$ for sufficiently  large $\rho$,
 which is automatically satisfied   when $n=3$.
 \begin{enumerate}
\item[(a)]  Suppose $ \tau > n-1$.   If $n=3$, also assume the decay conditions as in Lemma \ref{l-AH-4}.
Then, for each $0 \le i \le n$,
$$ \mathbf{M}(g) ( V^{(i)} ) =  2 b_n \lim_{\rho \rightarrow \infty} \int_{\Sigma_\rho} (H_0 - H) V^{(i)} \, d \sigma_\rho , $$
where $H$, $H_0$ are the mean curvature of $\Sigma_\rho$  in $(M^n, g)$, $\H^n$ respectively.

\item[(b)] Suppose $ \tau > n -1$ for $ n \ge 4 $ and $\tau>\frac52$ if $n=3$.
  If $n=3$, also assume the decay conditions as in Lemma \ref{l-AH-4}.
Suppose $\H^n$ is embedded in $\R^{n,1}$ as in \eqref{e-Hn},
then there exist isometric embeddings  of $(\Sigma_\rho, g_\rho )$ in $\H^n$ such  that
$$
\mathbf{M}(g)(V)=
2 b_n \lim_{\rho \to \infty}\int_{\Sigma_\rho }(H_0-H)\mathbf{x} \ d\sigma_\rho ,
$$
where
$ \mathbf{M}(g) (V)=\lf(\mathbf{M}(g) (V^{(0)}), \dots,\mathbf{M}(g) (V^{(n)})\ri)$
and $\mathbf{x} $ is   the position vector  of  the embedding  of $\S_\rho$  in $ \H^{n} \subset \R^{n,1}$.
\end{enumerate}
\end{enumerate}

\end{thm}

We will prove Theorem \ref{t-AH-1} via a series of lemmas. First,  we  fix  some notations.
Let $h_0$ be  the standard metric on the unit sphere $ \mathbb{S}^{n-1}$.
Choose    finitely many geodesic balls of radius $2$ in  $(\mathbb{S}^{n-1}, h_0)$,
so that the geodesic balls of  radius $1$ with the same centers cover $\mathbb{S}^{(n-1)}$.
In each geodesic ball, fix an orthonormal frame  $ \{ f_i  \}_{1 \le i  \le n-1}$.
Let $\{ \ve_i \}_{0 \le i \le n-1}$ be given as in Definition \ref{d-AH}.
Let $\nabla$, $\wt\nabla$ denote the covariant derivatives with respect to $g$, $g_0$ respectively,
where $g_0$ is the hyperbolic  metric on  $\H^n$.
Let $g_{ij} = g(\ve_i, \ve_j)$, ${g_0}_{ij} = g_0 ( \ve_i, \ve_j)$, and define
$ \{ \Gamma_{ij}^k \} $ and   $ \{ \wt\Gamma_{ij}^k \} $  by:
\be \label{e-def-gamma-t-gamma}
\nabla_{\ve_i}\ve_j=\Gamma_{ij}^k \ve_k, \ \
\wt\nabla_{\ve_i}\ve_j=\wt\Gamma_{ij}^k \ve_k.
\ee
Direct computation gives
  \be \label{e-gamma-tgamma}
\begin{split}
\Gamma_{ij}^k= & \ \frac12 g^{kl}\lf[ \ve_i(g_{jl}) +\ve_j(g_{il})  -\ve_l(g_{ij}) \ri. \\
& \ \lf. - g([\ve_i,\ve_l], \ve_j)-g([\ve_j,\ve_l],\ve_i)+ g([\ve_i,\ve_j],\ve_l)\ri], \\
\wt\Gamma_{ij}^k= & \ \frac12 \delta^{kl}   \lf[- g_0([\ve_i,\ve_l ], \ve_j)-g_0([\ve_j,\ve_l ],\ve_i)+ g_0([\ve_i,\ve_j],\ve_l)\ri] ,
\end{split}
\ee
for $ 0 \le i, j, k \le n-1$;
\be \label{e-Lie}
[\ve_0,\ve_i]=-\frac{\cosh \rho}{\sinh \rho}\ve_i, \  \ [\ve_i,\ve_j]=\frac1{ (\sinh \rho) }\sum_{k=1}^{ n-1} \lambda_{ij}^k\ve_k ,
\ee
for $1\le i, j\le n-1$, where $ \{ \lambda_{ij}^k \} $ are smooth functions in the geodesic balls of radius $2$ in $(\mathbb{S}^{n-1}, h_0)$,
and
\be  \label{e-R-tR}
\begin{split}
R(\ve_i,\ve_j,\ve_k,\ve_l)= & \ g_{pl}\lf\{\ve_i(\Gamma_{jk}^p)- \ve_j(\Gamma_{ik}^p) \ri. \\
& \ \lf. +\Gamma_{jk}^q\Gamma_{iq}^p-\Gamma_{ik}^q\Gamma_{jq}^p- (\Gamma_{ij}^q-\Gamma_{ji}^q)\Gamma_{qk}^p\ri\}\\
\wt R(\ve_i,\ve_j,\ve_k,\ve_l)= & \ \delta_{pl} { \lf\{\ve_i(\wt\Gamma_{jk}^p)- \ve_j(\wt\Gamma_{ik}^p) \ri. } \\
& \  \lf. +\wt\Gamma_{jk}^q\wt\Gamma_{iq}^p -\wt\Gamma_{ik}^q\wt\Gamma_{jq}^p - (\wt\Gamma_{ij}^q-\wt\Gamma_{ji}^q)\wt\Gamma_{qk}^p\ri\} ,
\end{split}
\ee
where $R$ and $\wt R$ are  the curvature tensors of $g$ and $g_0$ respectively,
It follows from \eqref{e-gamma-tgamma} -- \eqref{e-R-tR} that
\be \label{e-d-gamma-tgamma}
|\Gamma_{ij}^k-\wt\Gamma_{ij}^k|=O(e^{-\tau\rho})
\ee
and
\be \label{e-d-R-tR}
|R(X,Y,Z,W)-\wt R(X,Y,Z,W)|\le O(e^{-\tau\rho})|X|_{  g_0} |Y|_{ g_0} |Z|_{ g_0}  |W|_{  g_0}
\ee
for tangent vectors $X, Y, Z, W$.

\begin{remark} \label{rmk-grho}
 Because $\{ \ve_1, \ldots, \ve_{n-1} \}$ are tangential to $ \Sigma_\rho$, if we restrict all indices used in \eqref{e-def-gamma-t-gamma} --
\eqref{e-d-R-tR} to
the set $\{ 1, \ldots, n-1 \}$, then we obtain comparison between corresponding quantities on $(\Sigma_\rho, g_\rho)$ and $ (\Sigma_\rho, g_\rho^0)$. Here
$g_\rho$, $g_\rho^0$ denote the induced metric on $\Sigma_\rho$ from $g$, $g_0$ respectively.
\end{remark}

\begin{lma}\label{l-AH-2}
Let $\nu, \wt\nu$ be the unit outward normals to $ \Sigma_\rho$  and $A, \wt A$ be the second fundamental forms of  $\Sigma_\rho$ with respect to $g, g_0$ respectively. Then
\begin{enumerate}
  \item [(i)] $|\nu-\wt\nu|_{g_0}=O(e^{-\tau\rho})$.
  \item[(ii)] $\wt A(\ve_i,\ve_j)= (\coth\rho ) \delta_{ij}$, for $1\le i, j\le n-1$.
  \item [(iii)] $|A-\wt A|_{g_0}=O(e^{-\tau\rho}).$
\end{enumerate}
\end{lma}
\begin{proof} (i) This follows from the fact
 $\nu=\frac{\nabla\rho}{|\nabla \rho|_g}$ and
$$
\nabla \rho= g^{ij}\ve_j(\rho)\ve_i =g^{i0}\ve_i
= \ve_0+O(e^{-\tau\rho})
= \wt\nu+O(e^{-\tau\rho}).
$$

(ii) For $1\le i\le n-1$, by \eqref{e-gamma-tgamma},
$$
\wt A(\ve_i,\ve_j)= -g_0(\wt\nabla_{\ve_i}\ve_j,\ve_0)
=-\wt \Gamma_{ij}^0
= (\coth\rho)  \delta_{ij}.
$$

(ii)
For $1\le i, j \le n-1$, by \eqref{e-gamma-tgamma} and \eqref{e-d-gamma-tgamma},
\bee
\begin{split}
& \ A( \ve_i,\ve_j )-\wt A( \ve_i , \ve_j ) \\
= & \ - g (\nabla_{\ve_i}  \ve_j ,\nu) + g_0(\wt\nabla_{\ve_i} \ve_j ,\wt\nu)\\
=& \ - g(\nabla_{\ve_i}  \ve_j -\wt\nabla_{\ve_i} \ve_j ,\nu) - (g-g_0)(\wt\nabla_{\ve_i} \ve_j ,\nu) \\
& \ - g_0(\wt\nabla_{\ve_i} \ve_j, \nu-\wt\nu)\\
=& O(e^{-\tau\rho}).
\end{split}
\eee
\end{proof}

\begin{lma}\label{l-AH-3}
 Let $ \{ \kappa_i\}_{1\le i\le n-1}$ and  $H$  be the principal curvatures and  the mean curvature of $\Sigma_\rho$
 in $(M^n, g)$ with respect to $ \nu$, respectively.
Let $K$  be the sectional curvature associated to  a tangent plane   in  $(\Sigma_\rho, g_\rho)$,
where $g_\rho$ is the induced metric from  $g$. Then
\bee
\left\{
  \begin{array}{ll}
    \kappa_i=& \ \coth\rho+O(e^{-\tau\rho})  \\
    H=& \ (n-1)\coth\rho+O(e^{-\tau\rho})  \\
    K=& \
   (\sinh \rho)^{-2}   +O(e^{-\tau\rho}) .
  \end{array}
\right.
\eee
\end{lma}
\begin{proof} This  follows  from Lemma \ref{l-AH-2},  \eqref{e-d-R-tR} and the Gauss equation.
(The estimate  on $K$ can also be obtained by applying Remark \ref{rmk-grho}.)
 \end{proof}

\begin{lma}\label{l-AH-4}
Suppose $(\Sigma_\rho,g_\rho)$ can be isometrically embedded in $\H^n$ for large $\rho$.
Let $\{\ko_i\}_{1\le i\le n-1}$ be the principal curvatures of $\Sigma_\rho\subset \H^n$
with respect to the outward unit normal $\nu_0$.
 Then, for $n\ge4$,
$$\ko_i=\coth\rho+O(e^{-\tau\rho}) .$$
When $n=3$,  in addition to the asymptotic conditions specified in Definition \ref{d-AH}, if
for $1\le i, j, k, l, p, q\le 2$,
\be \label{e-h-derivative-condition}
\ve_p(\ve_k(\ve_l (g_{ij})))=O(e^{-   (\tau + 1)} \rho  ), \ve_q (\ve_p(\ve_k(\ve_l (g_{ij})))) =O(e^{-   (\tau +2)} \rho ),
\ee
 then the same conclusion also holds.
\end{lma}

\begin{proof}
  Let $H_0$, $A_0$  be the mean curvature, the second fundamental form  of $\S_\rho \subset \H^n $
  with respect to $\nu_0$,  respectively.
  At a point $p$  in $ \Sigma_\rho$, let $ \{ e_i\}_{1\le i\le n-1}$ be an orthonormal frame in $ T_p \Sigma_\rho $
  that diagonalizes $ A_0$.   Let $R^{\rho}_{ijkl}$, $R^{\rho}_{ij}$ and ${S}^\rho$  be the intrinsic curvature tensor,
  the Ricci tensor and the scalar curvature of $(\Sigma_\rho, g_\rho)$.  By the Gauss equation and Lemma \ref{l-AH-3},
\be\label{e-AH-estimate-1}
\left\{
  \begin{array}{rl}
   \ko_i\ko_j&=R^{\rho}_{ijij}+1=\coth^2\rho+O(e^{-\tau\rho}), \ i \neq j \\
    R^{\rho}_{ii}&=(n-2)   (\sinh \rho)^{-2}   +O(e^{-\tau\rho}) \\
    {S}^\rho&=(n-1)(n-2)  (\sinh \rho)^{-2}  +O(e^{-\tau\rho})
  \end{array}
\right.
\ee
and
\bee
\begin{split}
H_0^2= & \  (n-1)(n-2) + {S}^\rho+|A_0|^2  \\
\ge & \  (n-1)(n-2)\coth^2\rho+\frac1{n-1}H_0^2+O(e^{-\tau\rho}).
\end{split}
\eee
In particular, $H_0\neq 0$ and hence $ H_0 > 0 $ for large $\rho$ (as $\Sigma_\rho$ is compact).
Therefore,
\be\label{e-AH-estimate-2}
H_0\ge (n-1)\coth\rho+O(e^{-\tau \rho}).
\ee

Next, we  consider the case $n\ge 4$ first. The proof is similar to that of Lemma \ref{lma-embedding-AF}.
 By the Gauss equation,
 \bee
 R^{\rho}_{ii}+(n-2)=H_0\ko_i -(\ko_i)^2 ,
 \eee
 which shows
$$
\ko_i=\frac12\lf[ H_0\pm \sqrt { H_0^2-4  (R^{\rho}_{ii}+(n-2) ) } \ri].
$$
In particular, $\ko_i { > }  0$.
 We claim
$$
\ko_i=\frac12\lf[ H_0  -  \sqrt { H_0^2-4  (R^{\rho}_{ii}+(n-2) ) } \ri].
$$
If not,   without loss of  generality, suppose
$$
\kappa_1^{(0)}=\frac12\lf[ H_0 +  \sqrt { H_0^2-4  (R^{\rho}_{ii}+(n-2) ) } \ri] .
$$
By \eqref{e-AH-estimate-1} and \eqref{e-AH-estimate-2},
\bee
\begin{split}
& \ H_0^2-4(R^{\rho}_{ii}+(n-2)) \\
\ge &\ (n-1)^2\coth^2\rho- 4(n-2)\coth^2\rho+O(e^{-\tau\rho})\\
 =& \ (n-3)^2\coth^2\rho+O(e^{-\tau\rho}) >0
\end{split}
\eee
for large $\rho$, as   $n\ge4$.
Hence $\ko_1>\frac 12H_0$.
This implies, for $i\ge 2$,
\be\label{e-AH-estimate-3}
\begin{split}
  \kappa_i^{(0)}=&\ \frac12\lf[ H_0  -  \sqrt { H_0^2-4  (R^{\rho}_{ii}+(n-2) ) } \ri] \\
=&\frac12\frac{ 4(R^{(\rho)}_{ii}+(n-2))}{H_0+  \sqrt{ H_0^2-4(R^{\rho}_{ii}+(n-2))}} \\
  \le &\coth\rho+O(e^{-\tau\rho}),
  \end{split}
  \ee
for   otherwise we would have $\ko_1+\ko_i>H_0$, for some $ i \neq 1$, which is impossible.

Now for $i\neq j$ both larger than $1$, by \eqref{e-AH-estimate-1}  and  \eqref{e-AH-estimate-3},
\bee
\begin{split}
 \kappa_i^{(0)}\lf(\coth\rho+O(e^{-\tau\rho} ) \ri)\ge& \kappa_i^{(0)}\kappa_j^{(0)}
= \coth^2\rho+O(e^{-\tau\rho}) .
\end{split}
\eee
  Hence,
  \bee
 \kappa_i^{(0)}\ge \coth\rho+O(e^{-\tau\rho}).
  \eee
Combining this with \eqref{e-AH-estimate-3},   we have for $i\ge 2$,
  \be\label{e-AH-estimate-4}
  \kappa_i^{(0)}=\coth\rho+O(e^{-\tau\rho}).
  \ee
But then it follows from \eqref{e-AH-estimate-1}, \eqref{e-AH-estimate-2} and \eqref{e-AH-estimate-4} that, for $ i \ge 2$,
  \bee
\begin{split}
 \coth^2\rho+O(e^{-\tau\rho})= & \ \kappa_i^{(0)}\kappa_1^{(0)}\\
\ge & \ \lf(\coth\rho+O(e^{-\tau\rho})\ri)\frac{H_0}2\\
=& \ \frac{(n-1)}2\coth^2\rho+O(e^{-\tau \rho}) ,
\end{split}
  \eee
 which  is impossible since $n-1\ge 3$. Hence, \eqref{e-AH-estimate-3} is valid  for all $i \ge 1$.
 Repeating the argument leading to \eqref{e-AH-estimate-4} for any $ i \neq j$,
 we conclude that \eqref{e-AH-estimate-4} holds  for  all $i \ge 1$. This  proves the case  $n\ge 4.$

When  $n=3$,
  using the  assumption \eqref{e-h-derivative-condition}  and the fact $  \Gamma^{ k}_{ij} = O (e^{-\rho})$,
 for all $ 1 \le i, j, k \le 2 $, one  checks that
\be
\Delta_{g_\rho} {S}^\rho=O(e^{- (\tau + 2) \rho}),
\ee
where $ \Delta_{g_\rho}$ is the Laplacian on $ (\Sigma_\rho, g_\rho)$.
Then by {\cite[Lemma 2.3]{KwongTam}}  (see also \cite{LW}), we have
\bee
\begin{split}
H_0^2\le & \  {  \max_{\Sigma_\rho} \lf( \frac{2 ( {S}^\rho + 2 )^2 - 4 ( {S}^\rho + 2)
- \Delta_{g_\rho} {S}^\rho}{ {S}^\rho} \ri) } \\
= & \ 4\coth^2\rho+O(e^{-\tau\rho}).
\end{split}
\eee
Hence,
$$
H_0\le 2 \coth\rho+O(e^{-\tau\rho}),
$$
which together with \eqref{e-AH-estimate-2} shows
\be \label{e-AH-estimate-2dH0}
H_0= 2 \coth\rho+O(e^{-\tau\rho}).
\ee
Combining \eqref{e-AH-estimate-2dH0} with  \eqref{e-AH-estimate-1}, we conclude
$$\ko_\alpha =\coth\rho+O(e^{-\tau\rho}), \ \forall \ \alpha = 1, 2 .$$
This completes the proof.
\end{proof}

On the ball  model $({B}^n, g_0) $ of $\H^n$,
the vector fields
$ X^{(i)} $, $ 0 \le i \le n $, defined by
$$ X^{(0)} = x^j \frac{\p}{\p x^j} , \ \
X^{(k)}=\frac{\p}{\p x^k}, \ k=1,\dots, n , $$
are conformal Killing vector fields. Let $ r = | x | $ and define
$$   \rho =  \ln\lf(\frac{1+r}{1-r}\ri) ,$$
then    $ g_0 = d \rho^2 + (\sinh \rho)^2 h_0 $.
In particular, this shows
\be
 X^{(0)} = \sinh \rho \frac{\p}{\p \rho} , \ \ V^{(k)} = ( \sinh \rho) r^{-1} x^k , \ k=1,\dots, n ,
\ee
Recall that $V^{(0)}$ and $ V^{(k)}$ are defined in \eqref{e-def-Vi}.

\begin{lma}\label{l-AH-4-5}
Let $ | \Sigma_\rho | $ denote the volume of $(\Sigma_\rho, g_\rho)$. As $ \rho \rightarrow \infty$,
\bee\label{e-X0-1-1}
\begin{split}
& \  \int_{\Sigma_\rho} G^g_{-1} (X^{(0)},\nu)d\sigma_\rho  \\
 = & \ \frac12 \lf( \frac{ | \Sigma_\rho |  }{ \omega_{n-1} } \ri)^\frac{1}{n-1}
 \int_{\Sigma_\rho}  \lf[ \frac{n-2}{n-1} H^2- {S}^\rho -  (n-1)(n-2) \ri]  d\sigma_\rho+  o (1) .
 \end{split}
 \eee
Here   $d\sigma_\rho$ is   the volume  element of $g_\rho$ on $\Sigma_\rho$.
\end{lma}
\begin{proof}
Denote $G^g_{-1} $ by $G$.
As $(M^n, g)$ is asymptotically hyperbolic,
\be \label{eq-volume-S}
 | \Sigma_\rho |
= ( \sinh \rho )^{n-1} \omega_{n-1} ( 1 + O ( e^{- \tau \rho} ) )  .
\ee
Using  \eqref{e-d-R-tR}, Lemma  \ref{l-AH-2}  and the fact $\tau>\frac n2$,  as $\rho\to\infty$, we have
\be\label{e-X0-1-2}
\begin{split}
 \int_{\Sigma_\rho} G(X^{(0)},\nu)d\sigma_\rho
 = & \  \int_{\Sigma_\rho}\sinh\rho\  G(\wt \nu,\nu)d\sigma_\rho\\
 =& \ \int_{\Sigma_\rho}\sinh\rho\  G( \nu,\nu) d\sigma_\rho+o(1) .
 \end{split}
 \ee
By the Gauss equation,
\be \label{eq-gauss-S-rho}
\begin{split}
G (\nu, \nu)  = & \  \frac12 \lf[ H^2 - |A|^2 - {S}^\rho -  (n-1)(n-2) \ri] \\
 = & \  \frac12 \lf[ \frac{n-2}{n-1} H^2-  {S}^\rho -  (n-1)(n-2) \ri]  + O (e^{- 2 \tau \rho})   ,
\end{split}
\ee
where we also have used the fact
\bee
| A|^2 =\frac{1}{n-1} H^2 + O (e^{- 2 \tau \rho})
\eee
which follows from Lemma \ref{l-AH-3}.
Therefore,
\be\label{e-X0-1-3}
\begin{split}
& \  \int_{\Sigma_\rho} G(X^{(0)},\nu)d\sigma_\rho \\
=& \frac12 \sinh\rho  \int_{\Sigma_\rho}  \lf[ \frac{n-2}{n-1} H^2- {S}^\rho -  (n-1)(n-2) \ri]  d\sigma_\rho+o(1).
 \end{split}
 \ee
The lemma now  follows from \eqref{eq-volume-S} and \eqref{e-X0-1-3}.
\end{proof}

\begin{lma}\label{l-AH-5}
For each large $\rho$, suppose $ \iota_\rho: (\Sigma_\rho,g_\rho) \rightarrow  \H^n $ is an isometric embedding
  such  that its   principal curvatures $\{ \ko_i \}_{1\le i \le n-1}$ satisfy
$$
\ko_i=\coth\rho+O(e^{-\tau\rho}) ,
$$
 as $\rho\to\infty$. Then
\bee
\int_{\Sigma_\rho} (H_0-H)V^{(0)} d\sigma_\rho=-\frac1{n-2}\int_{\Sigma_\rho}G^g_{-1} (X^{(0)},\nu)d\sigma_\rho+o(1).
\eee
Here   $H_0$ is the mean curvature of the embedding $\iota_\rho$.

\end{lma}
\begin{proof}
Denote $G^g_{-1} $ by $G$.
Apply the Gauss equation to $\iota_\rho (\Sigma_\rho) \subset \H^n$, we have
\be \label{eq-gauss-S-rho-H}
\begin{split}
0 =  \sum_{i<j}( \ko_i \ko_j) - {S}^\rho -  (n-1)(n-2)  .
\end{split}
\ee
Therefore, by \eqref{e-X0-1-2}, \eqref{eq-gauss-S-rho} and \eqref{eq-gauss-S-rho-H},
\be\label{e-X0-1}
\begin{split}
 \int_{\Sigma_\rho} G(X^{(0)},\nu)d\sigma_\rho
 =& \  \int_{\Sigma_\rho}\sinh\rho\  G( \nu,\nu) d\sigma_\rho+o(1)\\
 =& \ \int_{\Sigma_\rho}\sinh\rho\sum_{i<j}( \k_i\k_j-\ko_i \ko_j)d\sigma_\rho+o(1) .
 \end{split}
 \ee
Now by the assumption  on $\ko_i$ and by Lemma \ref{l-AH-3},
\bee
 \begin{split}
 \kappa_i\kappa_j- \kappa_i^{(0)}\kappa_j^{(0)}=&\lf(\kappa_i-\kappa_i^{(0)}\ri)\kappa_j + \lf(\kappa_j-\kappa_j^{(0)}\ri)\kappa_i^{(0)}\\
 =&\lf(\kappa_i-\kappa_i^{(0)}\ri)\lf(\kappa_j-\coth\rho\ri) + \lf(\kappa_j-\kappa_j^{(0)}\ri)\lf(\kappa_i^{(0)}-\coth\rho\ri)\\
 &+ \coth\rho\lf(\kappa_i+\kappa_j-\ko_i-\ko_j\ri)\\
 =& \coth\rho\lf(\kappa_i+\kappa_j-\ko_i-\ko_j\ri)+O(e^{-2\tau\rho}) .
 \end{split}
 \eee
Hence,
$$
\sum_{i<j}( \k_i\k_j-\ko_i\ko_j)=(n-2)\coth\rho\,(H-H_0)+O(e^{-2\tau\rho}).
$$
Combining this with \eqref{e-X0-1}, we have
\bee
\begin{split}
\int_{\Sigma_\rho} G(X^{(0)},\nu)d\sigma_\rho=&-(n-2)\int_{\Sigma_\rho}\cosh\rho\, (H_0-H)d\sigma_\rho+o(1).
\end{split}
\eee
which proves the lemma.
\end{proof}

\begin{lma}\label{l-AH-6}
 With the same assumptions and notation as in Lemma \ref{l-AH-5}, suppose in addition $\tau>n-1$, then
\bee
\int_{\Sigma_\rho} (H_0-H)V^{(k)} d\sigma_\rho=-\frac1{n-2}\int_{\Sigma_\rho}G^g_{-1} (X^{(k)},\nu)d\sigma_\rho+o(1)
\eee
as $\rho\to\infty$,  $1\le k \le n$.
\end{lma}

\begin{proof}

We still denote $ G^g_{-1} $ by  $ G$. For a fixed $k$,  denote $X^{(k)}$ by $Y$.
Decompose  $ Y=Y_1+Z $,
where $Y_1= {r^{-2}} x^k X^{(0)}$ and  $Z$ is normal to $X^{(0)}$ with respect to $g_0$.
By the proof of Lemma \ref{l-AH-5}, we have

\be\label{e-Xp-1}
\begin{split}
 \int_{\Sigma_\rho} G (Y_1,\nu)d\sigma_\rho
=& \int_{\Sigma_\rho}\frac{x^k}{r^2}  G (X^{(0)},\nu)d\sigma_\rho\\
=&-(n-2)\int_{\Sigma_\rho}\frac{\theta^k}r\cosh\rho\, (H_0-H)d\sigma_\rho+o(1)\\
=&-(n-2)\int_{\Sigma_\rho} \frac{1+r^2}{2r^2} \theta^k \sinh\rho\, (H_0-H)d\sigma_\rho+o(1)\\
=&-(n-2)\int_{\Sigma_\rho} V^{(k)} (H_0-H)d\sigma_\rho+o(1)
 \end{split}
\ee
where we have used the assumption  $\tau>n-1$ and the fact
$$
 \lf(\frac{1+r^2}{2r^2}-1 \ri) \sinh\rho\, (H_0-H) =   \frac2 r  (H_0-H)=O(e^{-\tau\rho}),
$$
as $ \rho \rightarrow \infty$.

To estimate $ \int_{\Sigma_\rho} G(Z, \nu) d \sigma_\rho$, we adopt an argument in \cite{Herzlich}.
Let $\phi\ge0$ be a fixed smooth function on $[0,\infty)$
 such that $\phi(s)=0$ for $s\le \frac12$, $\phi(s)=1$ for $s\ge \frac 34$ and $0\le \phi\le 1$.
 For any  fixed  $\rho_0>0$ that is sufficiently large,   define
$$
  \wt g  = \lf[ 1-\phi(\frac{\rho}{\rho_0}) \ri] g_0+\phi(\frac{\rho}{\rho_0})   g =g_0+\phi ( \frac{\rho}{\rho_0}) (g-g_0),
$$
which is a metric on $ B^n$ that equals  $g$ outside $ \Omega_{\frac34 \rho_0}$ and agrees with
$ g_0$ inside $ \Omega_{\frac12 \rho_0} $. Here $ \Omega_\rho$ denotes the geodesic ball of radius $\rho$ in $(B^n, g_0)$ centered at the origin.
Let $\wt g_{ij}= \wt g(\ve_i,\ve_j)$, then
\bee \label{e-wtg-decay}
| \wt g_{ij}-\delta_{ij}|+|\ve_k( \wt g_{ij} )|+|\ve_k(\ve_l( \wt g_{ij}))|\le C_1e^{-\tau\rho_0}
\eee
on the annulus $ A_{ \rho_0} =  \Omega_{ \rho_0 }  \setminus \Omega_{ \frac14  \rho_0 }  $
for some constant $C_1$ that is  independent on $\rho_0$.
Hence,
\be \label{e-wtgG}
| G^{\wt g}_{-1}  |_{\wt g} \le C_2 e^{- \tau \rho_0}
\ee
on $A_{\rho_0}$ for some constant $ C_2$ independent on $\rho_0$.

Now let $ \wt \beta $ be the $1$-form dual to $ Z$ with respect to $ \wt g$.
Let $  (\wt \beta)^s  $ be  the symmetric  $(0,2)$  tensor given by
$ ( \wt \beta)^s_{ij} = \frac12 ( {\wt \beta}_{i;j} + {\wt \beta}_{j;i} ) $
where `` $;$ " denotes the covariant differentiation on $({B}^n, \wt g)$.
Integrating by parts and using the fact $G^{\wt g}_{-1} = 0 $ at $ \p \Omega_{\frac14 \rho_0} = S_{ \frac{1}{4} \rho_0}$, 
we have
\be \label{e-H-trick}
\int_{S_{\rho_0}} G (Z,\nu) d \sigma_{\rho_0}=
\int_{ A_{\rho_0} } \la G^{\wt g}_{-1},  ( \wt \beta)^s\ra_{\wt g}  dV_{\wt g} ,
\ee
where  $  d V_{\wt g}$ is the volume element of $\wt g$.

We next estimate $ | ( \wt \beta)^s |_{\wt g} $.
Note that if we write  $ Z = Z^i \ve_i $,   then
\be \label{e-Zi}
 | Z^i | + | \ve^j (Z^i) | \le C_3 e^{ \rho}
 \ee
for some constant $C_3$ independent on large $\rho$.
Let $\beta$ be the $1$-form dual to $Z$ with respect to $g_0$ and
let  `` $ { }_| $ " denote the covariant differentiation on $(B^n,  g_0)$.
On $ A_{\rho_0}$,  we have
\be\label{e-Xp-3}
\begin{split}
| \wt \beta_{i;j}-\beta_{i | j}|\le & \ |  \wt \beta_{i ; j}- \wt \beta_{i | j} |+|( \wt \beta -\beta)_{i | j}|\\
\le & \ C_4 e^{(-\tau+1)\rho_0}
\end{split}
\ee
for some constant $C_4$ independent on $\rho_0$,
where we have used  \eqref{e-Zi} and estimates analogues to
\eqref{e-gamma-tgamma} and   \eqref{e-d-gamma-tgamma}  with $ g$ replaced by $ \wt g$.
Let $ \beta^s $ be the symmetric $(0,2)$ tensor defined by
$ \beta^s_{ij} = \frac12 ( \beta_{i | j} + \beta_{j | i} ) $,
then direct  calculation gives 
\be\label{e-Xp-2}
\begin{split}
\beta^s_{ij}
= & \ \theta^k  \sinh\rho \, \delta_{ij}-\frac12\lf(\ve_i(\frac{x^k}{r^2})\delta_{j0}+
\ve_j(\frac{x^k}{r^2})\delta_{i0}\ri) -\frac{x^k}{r^2}\cosh\rho \, \delta_{ij}\\
=& \ \theta^k \lf(\sinh\rho-\frac1r\cosh\rho\ri)\delta_{ij}-\frac12\lf(\ve_i(\frac{x^k}{r^2})\delta_{j0}+
\ve_j(\frac{x^k}{r^2})\delta_{i0}\ri)\\
=& \ -\frac{\theta^k }{r}\delta_{ij}-\frac12\lf(\ve_i(\frac{x^k}{r^2})\delta_{j0}+
\ve_j(\frac{x^k}{r^2})\delta_{i0}\ri) .
\end{split}
\ee
Therefore,  by \eqref{e-Xp-3} and \eqref{e-Xp-2}, we conclude
\be \label{e-wtb}
| ({\wt \beta})^s |_{\wt g}  \le C_4
\ee
on $ A_{\rho_0}$  for some constant $C_4$ independent on $\rho_0$.
 It follows from \eqref{e-wtgG}, \eqref{e-H-trick} and \eqref{e-wtb} that
 $$
 \int_{S_{\rho_0}} G (Z,\nu) d \sigma_{\rho_0}= o(1)
 $$
 as $ \rho_0 \rightarrow \infty$ because $ \tau > n-1$.
This together with \eqref{e-Xp-1}  completes the proof.
\end{proof}

\begin{lma}\label{l-AH-4-6}
Suppose $\tau>n-1$. Then as $ \rho \rightarrow \infty$, for $ 1 \le k \le n $,
\bee\label{e-Xi-1-1}
\begin{split}
& \  \int_{\Sigma_\rho} G^g_{-1} (X^{(k)},\nu)d\sigma_\rho  \\
 = & \ \frac12 \lf( \frac{ | \Sigma_\rho |  }{ \omega_{n-1} } \ri)^\frac{1}{n-1}
 \int_{\Sigma_\rho}\frac{x^k}{|x|}  \lf[ \frac{n-2}{n-1} H^2- {S}^\rho -  (n-1)(n-2) \ri]  d\sigma_\rho+  o (1) .
 \end{split}
 \eee
 Here $\{ x^k \}$ are the coordinate functions in the ball model $(B^n, g_0)$ of $\H^n$.
\end{lma}
\begin{proof}
Since $ \tau > n-1$,
by the proof of Lemma \ref{l-AH-6}, we have
\be\label{e-Hawking-1}
\int_{\S_\rho}G(X^{(k)},\nu)d\sigma_\rho=  \int_{\S_\rho}\frac {x^k}{r^2}G(X^{(0)},\nu)d\sigma_\rho + o(1)
\ee
as $ \rho \rightarrow \infty$.
On the other hand, by  \eqref{eq-gauss-S-rho},
\bee \label{e-Hawking-2}
\begin{split}
& \  \int_{\S_\rho}\frac {x^k}{r^2}G(X^{(0)},\nu)d\sigma_\rho \\
= & \ \sinh \rho  \int_{\S_\rho}\frac {x^k}{r^2}G(\nu,\nu)d\sigma_\rho  + o(1) \\
=& \ \frac12\sinh\rho\int_{\S_\rho}  \frac {x^k}{r^2} \lf[ \frac{n-2}{n-1} H^2-  {S}^\rho -  (n-1)(n-2) \ri]  d\sigma_\rho+  o (1).
\end{split}
\eee
Since
$ \frac1{r^2}-\frac1r = O(e^{-\rho})$,
$$
  \frac{n-2}{n-1} H^2-  {S}^\rho -  (n-1)(n-2)=O(e^{-\tau\rho}) ,
  $$
  and $\tau>n-1$,
    the result follows from  \eqref{eq-volume-S} and \eqref{e-Hawking-1}.
\end{proof}

In the next Lemma, we normalize the embedding $\iota_\rho$ used in Lemma \ref{l-AH-5} and \ref{l-AH-6}
so that the corresponding position vector in $ \R^{n,1}$ can be compared to the vector
$ (V^{(0)}, V^{(1)}, \ldots, V^{(n)} ) $
defined at points in $\Sigma_\rho$.

\begin{lma}\label{l-AH-7}
With the same assumptions and notation as in Lemma \ref{l-AH-5}, with $\tau>n-1$,
 there exists  an isometry $\Lambda_\rho$ of $\H^n\subset \R^{n,1}$ such that the position vector $\Lambda_\rho\circ\iota_\rho(x)$ satisfies
$$|(\Lambda_\rho\circ\iota_\rho)^{(i)}(x)-V^{(i)}(x) |=O(e^{(-\tau+3)\rho})$$ for $0\le i\le n$ and for all $x\in \Sigma_\rho$.
Here $ (\Lambda_\rho\circ\iota_\rho)^{(i)}(x) $ is the $i$-th component of $ (\Lambda_\rho\circ\iota_\rho) (x)$
in $ \R^{n,1}$.

\end{lma}
\begin{proof}
We proceed as in \cite{KwongTam}.
Let $\coth \varsigma_1=\max_{x\in \Sigma_\rho, 1\le j  \le n-1}\ko_j (x)$,
where as before, $\ko_j$ are the principal curvatures of $\iota_\rho(\Sigma_\rho)$.
By Lemma \ref{l-AH-4}, for large $\rho$, we have
$$
\coth\varsigma_1=\coth\rho+O(e^{-\tau\rho})>1
$$
since $\tau>2$. Therefore,
\bee
\begin{split}
\varsigma_1= & \ \frac12\ln\lf[\frac{\coth\rho+O(e^{-\tau\rho})+1}{\coth\rho+O(e^{-\tau\rho})-1}\ri]\\
= & \ \rho+O(e^{(-\tau+2)\rho}).
\end{split}
\eee
Similarly, let   $\coth\varsigma_2=\min_{x\in \Sigma_\rho, 1\le j \le n-1 }\ko_j (x)$, then $\coth\varsigma_2>1$ and
$$
\varsigma_2= \rho+O(e^{(-\tau+2)\rho}).
$$

By \cite[Theorem 4.5]{H} and  \cite[Proposition 3.1 and 3.2]{KwongTam},
we conclude that $\iota_\rho(\Sigma_\rho)$ lies between two geodesic balls with the same center with radii $\varsigma_{\text {\bf e}}\ge \varsigma_{\text {\bf i}}$ so that
\be\label{e-AH-estimate-5-0}
\varsigma_{\text {\bf e}}, \  \varsigma_{\text {\bf i}}= \rho+O(e^{(-\tau+2)\rho}).
\ee
Now, first choose an isometry $\Lambda_\rho$ so that the center of these  geodesic balls is at $o=(1,0,0,0)\in \H^n\subset \R^{n,1}$. Let  $\varsigma $ be the hyperbolic distance function from $o$.
Given $ x \in \Sigma_\rho$, let  $y(x) =\Lambda_\rho\circ\iota_\rho(x)$,
then by \eqref{e-AH-estimate-5-0},
\be\label{e-AH-estimate-5}
\begin{split}
\sinh \varsigma(y (x)) = & \ \sinh\rho+O(e^{(-\tau+3)\rho}), \\
  \cosh \varsigma(y (x))= & \ \cosh\rho+O(e^{(-\tau+3)\rho}),
  \end{split}
\ee
as $\rho\to\infty$.
For the simplicity of notation, we still denote $\Lambda_\rho\circ\iota_\rho$ by $\iota_\rho$.
Using geodesic polar coordinates in $\H^n$ with respect to $o=(1,0,0,0)$, for each $x\in \Sigma_\rho$, we
define $ \theta (x), \vartheta (x)   \in \mathbb{S}^{n-1} \subset T_o \H^n = \R^n $ respectively by
$$ x=\exp_o(\rho\theta(x)) \ \
 \mathrm{and} \ \
y (x)=\exp_o(\varsigma( y(x) ) \vartheta(x)).$$
We want to estimate $\vartheta(x)-\theta(x)$ as points in $\R^n$.
Let $d_\rho$ be the intrinsic distance function on $(\Sigma_\rho,g_\rho)$, $\wt d_\rho$ be the intrinsic distance of 
$\iota_\rho (\Sigma_\rho)$, and let $d_{\mathbb{S}^{n-1}}$ be the distance function on $\mathbb{S}^{n-1}$ with respect to the standard metric $h_0$.
Given  $x_1, x_2\in \Sigma_\rho$, we have
\bee
\begin{split}
\sinh\rho  \,d_{\mathbb{S}^{n-1}}(\theta(x_1),\theta(x_2))= & \ d_\rho(x_1,x_2)\\
=& \ \wt d_\rho(y(x_1),y(x_2))\\
=& \ [ \sinh\rho+O(e^{(-\tau+3)\rho}) ]d_{\mathbb{S}^{n-1}}(\vartheta(x_1),\vartheta(x_2))
\end{split}
\eee
by \eqref{e-AH-estimate-5}. Hence,
\be\label{e-AH-estimate-6}
d_{\mathbb{S}^{n-1}}(\vartheta(x_1),\vartheta(x_2))
= [1+O(e^{(-\tau+2)\rho}) ] d_{\mathbb{S}^{n-1}}(\theta(x_1),\theta(x_2)).
\ee
Next,  let $e_1=(1,0,\dots,0),\dots, e_n=(0,0,\cdots,1)\in \mathbb{S}^{n-1}\subset \R^n$.
By composing $\iota_\rho$ with an isometry  of $\H^n$ fixing $o$ and still
 denoting the resulting composition by $\iota_\rho$,  we may assume
$ \vartheta(\rho e_1)=e_1$
and, for each $i>1 $,  $\vartheta(\rho e_i )$ is a linear combination of $\{ e_1,\dots, e_i \}$
such that the   coefficient of $e_i $ is nonnegative.
For this $\iota_\rho$, we  claim that for $1\le j \le n$,
\be\label{e-AH-estimate-7}
d_{\mathbb{S}^{n-1}}(\vartheta(\rho e_j),e_j)=  O(e^{(-\tau+2)\rho}).
\ee
This is obvious true for $j=1$ because $
\vartheta(\rho e_1)=e_1$. Suppose \eqref{e-AH-estimate-7} is true for
$j = 1, \ldots, i $ where $ i < n $.
Write
$$
\vartheta ( \rho e_{i+1})=\sum_{j=1}^{i+1}a_je_j
$$
with $a_{i+1}\ge 0$.
For $1\le j\le i$, by the triangle inequality,
\bee
| d_{\mathbb{S}^{n-1}}(\vartheta(\rho e_{i+1}),e_j)- d_{\mathbb{S}^{n-1}}(\vartheta(\rho e_{i+1}),\vartheta(\rho e_j)|\le
d_{\mathbb{S}^{n-1}}(\vartheta(\rho e_j),e_j).
\eee
 Hence,  by \eqref{e-AH-estimate-6},  we have
 \be
 d_{\mathbb{S}^{n-1}}(\vartheta(\rho e_{i+1}),e_j)=\frac\pi2+  O(e^{(-\tau+2)\rho}),
 \ee
 which shows
 $$
 a_j=\cos (d_{\mathbb{S}^{n-1}}(\vartheta(\rho e_{i+1}),e_j))= O(e^{(-\tau+2)\rho})
 $$
 for all $j  \in \{1, \ldots , i \}$.  Thus,
 $$
 a_{i+1}=\lf(1-\sum_{j=1}^2a_j^2\ri)^\frac12=1+O(e^{(-\tau+2)\rho})
 $$
 because $a_{i+1}\ge 0$. Therefore,  \eqref{e-AH-estimate-7} holds  for $j = i+1$
 and hence it is true  for all $1\le j \le n$.
 Now let $x\in \Sigma_\rho$,   by  \eqref{e-AH-estimate-7} and \eqref{e-AH-estimate-6},
 \bee
\begin{split}
 d_{\mathbb{S}^{n-1}}(\vartheta(x),e_j )= & \ d_{\mathbb{S}^{n-1}}(\vartheta(x),  \vartheta( \rho e_j))+O(e^{(-\tau+2)\rho})\\
 =& \ d_{\mathbb{S}^{n-1}}(\theta(x), e_j)+O(e^{(-\tau+2)\rho}).
\end{split}
\eee
Hence, this implies
\be\label{e-AH-estimate-8}
\begin{split}
\vartheta(x)-\theta(x)=& \
\sum_{i=1}^n \lf[ \cos(d_{\mathbb{S}^{n-1}}(\vartheta(x),e_i))- \cos(d_{\mathbb{S}^{n-1}}(\theta(x), e_i))\ri]e_i\\
= & \ O(e^{(-\tau+2)\rho}).
\end{split}
\ee
The lemma now follows from \eqref{e-AH-estimate-5},  \eqref{e-AH-estimate-8} and the fact
$$ V^{(0)} (x) = \cosh \rho ,   \ V^{(i)}  (x) =\theta^{(i)}(x)\sinh \rho, $$
$$ ( \iota_\rho )^{(0)} (x) =\cosh\varsigma ( y(x) ),   \  ( \iota_\rho)^{(i)}  (x) =\vartheta^{(i)}(x)\sinh\varsigma( y(x) )$$
for $ i \ge 1$.
\end{proof}

Now we are ready to prove Theorem \ref{t-AH-1}.

\begin{proof}[Proof of Theorem \ref{t-AH-1}]
 (a) and (b) of  Part (I)  are  direct consequences of Lemma \ref{l-AH-4-5}, \ref{l-AH-4-6}
 and Theorem \ref{t-mass} (ii).

As for Part (II), (a)
follows  from  Lemma \ref{l-AH-4}, \ref{l-AH-5}, \ref{l-AH-6} and Theorem \ref{t-mass} (ii).
To prove   (b),
by Lemma \ref{l-AH-4},  \ref{l-AH-5}, \ref{l-AH-6}  and \ref{l-AH-7},
 there exist isometric embeddings
 $\iota_\rho = (\mathbf{x}^0, \ldots, \mathbf{x}^n ):  (\Sigma_\rho, g_\rho ) \rightarrow  \H^n \subset \R^{n,1}$
 such that, for each $ i =0, \ldots, n$,
 \be \label{eq-final}
 \begin{split}
& \  \int_{\Sigma_\rho} ( H_0 - H) \mathbf{x}^i \, d \sigma_\rho \\
 =  & \  \int_{\Sigma_\rho} ( H_0 - H) V^{(i)} \, d \sigma_\rho  +  \int_{\Sigma_\rho} ( H_0 - H) ( \mathbf{x}^i - V^{(i)})  \, d \sigma_\rho \\
 = & \ - \frac{1}{n-2} \int_{\Sigma_\rho} G^g_{-1} (X^{(i)}, \nu)  \, d \sigma_\rho + o(1)
 \end{split}
 \ee
 provided $ 2\tau -3  > n-1$. But this is satisfied if  $ \tau > n-1 $ when $ n\ge 4$ and
 $ \tau > \frac{5}{2} $ when $ n =3$.
 Hence,  (b)  follows from \eqref{eq-final} and Theorem \ref{t-mass} (ii).
\end{proof}



\begin{thebibliography}{1000}

\bibitem{ADM61} Arnowitt,  R.;  Deser, S., and Misner,  C. W.,
{\sl Coordinate invariance and energy expressions in general relativity}, Phys. Rev. (2) \textbf{122} (1961), 997--1006.

\bibitem{Ashtekar-Hansen} 
Ashtekar, A.;  Hansen, R. O., 
{\sl A unified treatment of null and spatial infinity in general relativity. I. Universal structure, asymptotic symmetries, and conserved quantities at spatial infinity}, J. Math. Phys. \textbf{19} (1978), no. 7, 1542--1566. 


\bibitem{BY1} Brown, J. D.; York, J. W., Jr.,
{\sl Quasilocal energy in general relativity},
in {Mathematical aspects of classical field theory (Seattle, WA,   1991)},
volume 132 of {Contemp. Math.}, pages 129--142. Amer. Math. Soc., Providence, RI, 1992.

\bibitem{BY2}
Brown, J. D.; York, J. W., Jr.,
{\sl  Quasilocal energy and conserved charges derived from the   gravitational action},
{Phys. Rev. D (3)}, 47(4):1407--1419,1993.

\bibitem{Chrusciel}
Chru\'sciel, P.,
{\sl A remark on the positive energy theorem}, Class. Quantum Grav.  \textbf{3} (1986), L115--L121.

\bibitem{ChruscielHerzlich}
  Chru\'sciel, P.; Herzlich, M.,
{\sl The mass of asymptotically hyperbolic Riemannian manifolds}, Pacific J. Math. {\bf 212} (2003), no. 2, 231--264.

\bibitem{FanKwong} Fan, X.-Q.; Kwong, K.-K., {\sl A property of the Brown-York mass in Schwarzschild manifolds}, J. Math. Anal. Appl.  \textbf{400}  (2013),  no. 2, 615--623.

\bibitem{FanShiTam}
Fan, X.-Q., Shi, Y.-G.; Tam, L.-F., {\sl Large-sphere and small-sphere limits of the Brown-York mass},
Comm. Anal. Geom. \textbf{17}(2009), no. 1, 37--72.

\bibitem{Hawking-mass}  Hawking, S.W.,
{\sl Gravitational radiation in an expanding universe}, J. Math. Phys. \textbf{9} (1968), 598--604.

\bibitem{Herzlich} Herzlich, M., {\sl Computing asymptotic invariants with the Ricci tensor on asymptotically flat and hyperbolic manifolds}, arXiv:1503.00508.

\bibitem{H}  Howard, R.,
{\sl Blaschke's rolling theorem for manifolds with boundary}, Manuscripta Mathematica
  \textbf{99} (1999), no. 4, 471--483.

\bibitem{Huang-ICCM} Huang, L.-H.,
{\sl On the center of mass in General Relativity,} in Fifth International Congress of Chinese mathematicians,
Part I, 2, AMS/IP Stud. Adv. Math., vol. 51, Amer. Math. Soc., Providence, RI, 2012, 575--591.

\bibitem{Huisken1986}   Huisken, G., {\sl
Contracting convex hypersurfaces in Riemannian manifolds by their mean curvature},
Invent. Math. \textbf{84} (1986),  no. 3, 463--480.

\bibitem{KLBG}
Klingenberg, W., {\sl  A course in differential geometry},
Translated from the German by David Hoffman. Graduate Texts in
Mathematics, Vol. \textbf{51}, Springer-Verlag, New
York-Heidelberg, 1978.

\bibitem{KwongTam}  Kwong, K.-K.;  Tam, L.-F., {\sl  Limit of quasilocal mass integrals in asymptotically hyperbolic manifolds}, Proc. Amer. Math. Soc. \textbf{141}(2013), 313--324.

\bibitem{LW}
Li, Y.; Weinstein, G., {\sl A priori bounds for co-dimension one isometric embeddings}, Amer. J. Math. \textbf{121}(1999), no. 5, 945--965.

\bibitem{MiaoTam2015} Miao, P.; Tam, L.-F., {\sl Evaluation of the ADM mass and center of mass via the Ricci tensor},
Proc. Amer. Math. Soc. \textbf{144} (2016), no 2, 753--761.


\bibitem{Michel}
Michel, B., {\sl Geometric invariance of mass-like invariants}, J. Math. Phys. \textbf{52} (2011), 052504.


\bibitem{Neves} Neves, A.,
{\sl Insufficient convergence of inverse mean curvature flow on asymptotically hyperbolic manifolds},
J. Differential Geom.  \textbf{84}  (2010), no. 1, 191--229.

  \bibitem{Nirenberg}
Nirenberg, L., {\sl The Weyl and Minkowski problem in differential geometry in the large},
Commun. Pure Appl. Math. \textbf{6} (1953), no. 3, 337--394.

 \bibitem{P} Pogorelov,   A.,
{\sl Some results on surface theory in the large}, Advances in Math.
  \textbf{1} (1964), no. 2, 191--264.


\bibitem{ShiTam02} Shi, Y.-G.;  Tam,  L.-F.,
{\sl Positive mass theorem and the boundary behaviors of compact manifolds with nonnegative scalar curvature},
J. Differential Geom. \textbf{62}  (2002),  no.1, 79--125.


\bibitem{ShiTam2007}
Shi, Y.-G.; Tam,  L.-F.,
{\sl Rigidity of compact manifolds and positivity of quasi-local mass}, Class. Quantum Grav.,  {\bf 24} (2007), no. 9, 2357--2366.

\bibitem{ShiWangWu09} Shi, Y.-G.; Wang, G.; Wu, J.,
{\sl On the behavior of quasi-local mass at the infinity along nearly round surfaces},
Ann. Glob. Anal. Geom, \textbf{36}(2009), no.4, 419--441.


\bibitem{WangWu2015} Wang, G.; Wu, J.,
{\sl Chern's magic form and the Gauss-Bonnet-Chern mass},
arXiv:1510.03036.


\bibitem{WangYau2007}
 Wang, M.-T.; Yau, S.-T.,
{\sl A generalization of Liu-Yau's quasi-local mass}, Comm. Anal. Geom. \textbf{15} (2007), no. 2, 249--282.

\bibitem{Wang}
Wang, X., {\sl The mass of asymptotically hyperbolic manifolds}, J. Differential Geom.  \textbf{57} (2001), no. 2, 273--299.

\bibitem{Zhang}
Zhang, X., {\sl A definition of total energy-momenta and the positive mass theorem on asymptotically hyperbolic 3 manifolds I},  Commun. Math. Phys.  \textbf{249} (2004), no. 3, 529--548.


\end{thebibliography}
\end{document}